\documentclass[a4paper,12pt,draft]{amsart}
\usepackage[margin=30mm]{geometry}
\usepackage{amssymb, amsmath, amsthm, amscd}

\usepackage[T1]{fontenc}
\usepackage{eucal,mathrsfs,dsfont}
\usepackage{color}

\renewcommand{\le}{\leqslant}
\renewcommand{\ge}{\geqslant}

\definecolor{mno}{rgb}{0.5,0.1,0.5}

\newcommand{\R}{\mathds R}

\newcommand{\I}{\mathds 1}

\newtheorem{theorem}{Theorem}[section]
\newtheorem{lemma}[theorem]{Lemma}
\newtheorem{proposition}[theorem]{Proposition}
\newtheorem{corollary}[theorem]{Corollary}

\theoremstyle{definition}

\newtheorem{example}[theorem]{Example}
\newtheorem{remark}[theorem]{Remark}

\begin{document}
\allowdisplaybreaks
\title[Weighted Poincar\'{e}
Inequalities for Nonlocal Dirichlet Forms] {\bfseries
Weighted Poincar\'{e} Inequalities for Nonlocal Dirichlet Forms}

\author{Xin Chen \quad Jian Wang}
\thanks{\emph{X.\ Chen:}
   Grupo de Fisica Matematica, Universidade de Lisboa, Av Prof Gama Pinto 2, 1649-003 Lisbon, Portugal. \texttt{chenxin\_217@hotmail.com}}

  \thanks{\emph{J.\ Wang:}
   School of Mathematics and Computer Science, Fujian Normal University, 350007 Fuzhou, P.R. China. \texttt{jianwang@fjnu.edu.cn}}

\date{}

\maketitle

\begin{abstract}
Let $V$ be a locally bounded measurable function such that $e^{-V}$ is bounded and belongs to $L^1(dx)$, and let $\mu_V(dx):=C_V e^{-V(x)}\,dx$ be a probability measure. We
present the criterion for the weighted Poincar\'{e} inequality of
the non-local Dirichlet form
$$
    D_{\rho,V}(f,f):=\iint(f(y)-f(x))^2\rho(|x-y|)\,dy\, \mu_V(dx)
$$ on $L^2(\mu_V)$.
Taking $\rho(r)={e^{-\delta r}}{r^{-(d+\alpha)}}$ with
$0<\alpha<2$ and $\delta\geqslant 0$, we get some conclusions for
general fractional Dirichlet forms, which can be regarded as a
complement of our recent work \cite{WW}, and an improvement of the
main result in \cite{MRS}. In this especial setting, concentration of
measure for the standard Poincar\'{e} inequality is also derived.

Our technique is based on the Lyapunov conditions for the associated
truncated Dirichlet form,  and it is considerably efficient for the
weighted Poincar\'{e} inequality of the following non-local
Dirichlet form
$$
    D_{\psi,V}(f,f):=\iint(f(y)-f(x))^2\psi(|x-y|)\,e^{-V(y)}\,dy\,e^{-V(x)}\,dx
$$
on $L^2(\mu_{2V})$, which is associated with symmetric Markov
processes under Girsanov transform of pure jump type.
\medskip

\noindent\textbf{Keywords:} Non-local Dirichlet form; weighted
Poincar\'{e} inequality; Lyapunov conditions; concentration of
measure.

\medskip

\noindent \textbf{MSC 2010:} 60G51; 60G52; 60J25; 60J75.
\end{abstract}
\section{Introduction and Main Results}\label{sec1}
\subsection{Background for Functional Inequalities of Fractional Dirichlet Forms}
For $\alpha\in(0,2)$, let $\mu_\alpha$ be a rotationally symmetric
stable infinite divisible probability distribution, such that
$$
\hat{\mu}_\alpha(\xi):=\int e^{i x\cdot\xi}\,\mu_\alpha(dx) =
e^{-\frac{1}{\alpha}|\xi|^{\alpha}},\quad  \xi\in\R^d.
$$
For any $f\in C_b^\infty(\R^d)$, the set of smooth functions with
bounded derivatives of every order, define
\begin{equation}\label{sec1.1.1}D_\alpha(f,f):=\iint  \frac{(f(y)-f(x))^2}{|y-x|^{d+\alpha}}\,dy\,
\mu_\alpha(dx).\end{equation} Then, $(D_\alpha, C_b^\infty(\R^d))$
can be extended to a \emph{non-local Dirichlet form} associated with
the operator
$$L_\alpha=-(-\Delta)^{\alpha/2}-x\cdot\nabla,\qquad\alpha\in (0,2),$$ which is the
infinitesimal generator of an Ornstein-Uhlenbeck process driven by
symmetric $\alpha$-stable L\'{e}vy processes. Poincar\'{e}
inequalities for $(D_\alpha, C_b^\infty(\R^d))$ were studied in
\cite[Theorem 1.3 and Corolalry 1.4]{RW}.

\medskip

In \eqref{sec1.1.1} the singular kernel ${|y-x|^{-(d+\alpha)}}\,dy$ is the L\'{e}vy measure associated with $\mu_\alpha$, which is a strong constraint to study functional inequalities for general non-local Dirichlet forms. The first breakthrough in this direction was established in \cite{MRS} in virtue of the methods from harmonic analysis. The main result in \cite{MRS} (see \cite[Theorem 1.2]{MRS})
 states that, \emph{if $e^{-V}\in L^1(dx)\cap C^2(\R^d)$ such that for some constant
$\varepsilon>0$,
\begin{equation}\label{sec1.1.2}\frac{(1-\varepsilon)|\nabla V|^2}{2} -\Delta V\to\infty,\qquad
x\to \infty,\end{equation} then there exist two positive constants
$\delta$ and $C_0$ such that for all $f\in C_b^\infty(\R^d)$ with
$\int f(x)e^{-V(x)}\,dx=0$,
\begin{equation}\label{sec1.1.3}\aligned\int f^2(x)\big(1+|\nabla  V(x)|^{\alpha}\big) &e^{-V(x)}\,dx\\
\leqslant &C_0\iint
\frac{(f(y)-f(x))^2}{|y-x|^{d+\alpha}}e^{-\delta|y-x|}\,dy
e^{-V(x)}\,dx.\endaligned\end{equation}}

\medskip

On the other hand, as a generalization of \eqref{sec1.1.1}, recently explicit and sharp criteria of Poincar\'{e} type (i.e., Poincar\'{e},
super Poincar\'{e} and weak Poincar\'{e}) inequalities have been
presented in \cite{WW} for the following general fractional Dirichlet form
\begin{equation}\label{sec1.1.4}
D_{\alpha,V}(f,f):=\iint  \frac{(f(y)-f(x))^2}{|y-x|^{d+\alpha}}\,dy\,
\mu_V(dx),\quad d\ge1, \alpha\in (0,2),
\end{equation} where $V$ is a Borel measurable function on $\R^d$ such that $e^{-V}\in L^1(dx)$, and $\mu_V(dx)=\frac{1}{\int\,e^{-V(x)}\,dx}\,e^{-V(x)}\,dx$.
According to the paragraph below
\cite[Remark 1.3]{MRS}, \eqref{sec1.1.4} is natural in the sense that:
 we should regard the measure
${|y-x|^{-(d+\alpha)}}\,dy$ as the L\'{e}vy measure, and
$\mu_V(dx)$ as the ambient measure. Namely, $D_{\alpha,V}$ does get rid of the constraint in $D_\alpha$, and it should be a typical example in study functional inequalities for non-local Dirichlet forms.

To move further, we briefly recall the results developed in
\cite{WW}. Let $e^{-V}\in L^1(dx)\cap C^2(\R^d)$ satisfying some
regular assumptions. \cite[Theorem 1.1 (1) and (2)]{WW} shows
that,\emph{ if $$
\liminf_{|x|\to\infty}\frac{e^{V(x)}}{|x|^{d+\alpha}}>0,$$ then there
is a constant $C_1>0$ such that for all $f\in
C_b^\infty(\R^d)$,
\begin{equation}\label{sec1.1.5}\mu_V\big( f-\mu_V(f)\big)^2\leqslant C_1 D_{\alpha,V}(f,f);\end{equation} if $$
\liminf_{|x|\to\infty}\frac{e^{V(x)}}{|x|^{d+\alpha}}=\infty,$$ then the following super-Poincar\'{e} inequality
$$\mu_V(f^2)\le r D_{\alpha,V}(f,f)+\beta(r)\mu_V(|f|)^2,\quad r>0$$
holds for some non-increasing function $\beta$ and all $f\in C_b^\infty(\R^d)$.}

\medskip

Note that the exponentially decaying factor $e^{-\delta |y-x|}$ and
the weighted function $1+|\nabla  V(x)|^\alpha$ in \eqref{sec1.1.3}
indicate that the functional inequality \eqref{sec1.1.3} is stronger
than the expected Poincar\'{e} inequality \eqref{sec1.1.5} for fractional
Dirichlet form $D_{\alpha,V}$, see \cite[Remark 1.4]{MRS}.
Therefore, the work of \cite{WW} does not extend \cite{MRS}, and
there still exists a gap between \cite{WW} and \cite{MRS}. That is just
the motivation of our present paper.

\subsection{Weighted Poincar\'{e} Inequalities for $D_{\alpha,V,\delta}$ with $\delta>0$: Improvement of the Work in \cite{MRS}}
 We  first introduce some notations.
 Let $V$ be a locally bounded measurable function on $\R^d$ such that $e^{-V}$ is bounded and $e^{-V}\in
L^1(dx)$. Define a probability measure $\mu_V$ as
follows
\begin{equation}\label{sec1.1.6}
\mu_V(dx)=\frac{1}{\int\,e^{-V(x)}\,dx}\,e^{-V(x)}\,dx.
\end{equation} For any $\delta \ge 0$ and $f\in C_b^\infty(\R^d)$, set
$$D_{\alpha,V,\delta}(f,f):=\iint
\frac{(f(y)-f(x))^2}{|y-x|^{d+\alpha}}e^{-\delta|y-x|}\,dy
\,\mu_V(dx).$$ In particular, when $\delta=0$, $D_{\alpha,V,\delta}=D_{\alpha,V}$.
We say that the weighted Poincar\'{e} inequality holds for $D_{\alpha,V,\delta}$, if there exist a positive weighted function
$\tilde \omega$ and a constant $C>0$ such that for all $f\in C_b^\infty(\R^d)$,
$$\int\big(f-\mu_V(f)\big)^2\tilde \omega\,d\mu_V\le CD_{\alpha,V,\delta}(f,f).$$

Even though it is known that in the context of local
Dirichlet forms some super Poincar\'{e} inequalities can imply
weighted Poincar\'{e} inequalities (e.g.\ see \cite{W3}), to the best of our knowledge there is no literature about such relation for
non-local Dirichlet form $D_{\alpha,V,\delta}$,
even for $D_{\alpha,V}$. Instead of
studying this topic, the purpose of this paper is to establish the
weighted Poincar\'{e} inequalities for non-local Dirichlet form
$D_{\alpha,V,\delta}$ directly.

\bigskip

The main result is as following,

\begin{theorem}\label{thm1.1} Suppose that for some constants $\delta > 0$, $\alpha\in(0,2)$
and $\alpha_0\in(0,1),$
\begin{equation}\label{thm1.1.1}\limsup_{|x|\to\infty}
\bigg[\bigg(\sup_{|z|\geqslant |x|}e^{-V(z)}\bigg)e^{\delta |x|}|x|^{d+\alpha-\alpha_0}\bigg]=0.
\end{equation}
Then, there exists a constant $C_1>0$ such that the following
weighted Poincar\'{e} inequality
\begin{equation}\label{thm1.1.2}
\aligned
\int \big(f (x)-\mu_V(f)\big)^2&
\frac{e^{V(x)-\delta|x|}}{(1+|x|)^{d+\alpha}}\,\mu_V(dx)\leqslant C_1D_{\alpha,V,\delta}(f,f)\endaligned
\end{equation}
holds for all $f\in C_b^\infty(\R^d)$.
\end{theorem}

\ \

To see that Theorem \ref{thm1.1} improves \cite[Theorem
1.2]{MRS}, we consider the following example.

\begin{example}\label{exa1.2} (1)
For $\delta>0$, let $V(x)=\varepsilon(1+|x|^2)^{1/2}$ with some
$\varepsilon>\delta$. Then, \eqref{thm1.1.1} is fulfilled, and so
the corresponding weighted Poincar\'{e} inequality \eqref{thm1.1.2}
holds. Note that, \eqref{sec1.1.2} is not satisfied for
$V(x)=\varepsilon(1+|x|^2)^{1/2}$ with any $\varepsilon>0$.

(2) Let $V(x)=1+|x|^2$. Then, Theorem \ref{thm1.1} implies that
for any $\delta>0$ there exists a constant $c_1>0$ such that for all
$f\in C_b^\infty(\R^d)$ with $\int f(x)e^{-V(x)}\,dx=0$,
\begin{equation*}\aligned\int f^2(x) \exp\Big(\frac{1}{2}{\big(1+|x|^2\big)}\Big) &e^{-V(x)}\,dx
\le c_1\iint
\frac{(f(y)-f(x))^2}{|y-x|^{d+\alpha}}e^{-\delta|y-x|}\,dy
e^{-V(x)}\,dx.\endaligned\end{equation*} However, in this setting
\cite[Theorem 1.2]{MRS} only implies that there exist two constants
$\delta$, $c_0>0$ such that for all $f\in C_b^\infty(\R^d)$ with
$\int f(x)e^{-V(x)}\,dx=0$,
\begin{equation*}\aligned\int f^2(x) \big(1+|x|^{\alpha}\big) &e^{-V(x)}\,dx
\le c_0\iint
\frac{(f(y)-f(x))^2}{|y-x|^{d+\alpha}}e^{-\delta|y-x|}\,dy
e^{-V(x)}\,dx.\endaligned\end{equation*}
\end{example}

\bigskip

As a direct consequence of Theorem \ref{thm1.1}, we know that for
$\delta > 0$ and $\alpha\in(0,2)$, if
\begin{equation}\label{sec1.1.9}\liminf_{|x|\to\infty}\frac{e^{V(x)}}{|x|^{d+\alpha}e^{\delta|x|}}>0,\end{equation}
then \eqref{thm1.1.2} holds, which implies the standard Poincar\'{e}
inequality: \begin{equation}\label{sec1.1.10} \aligned \mu_V \big(f
-\mu_V(f)\big)^2 & \leqslant
C_2D_{\alpha,V,\delta}(f,f)\quad\textrm{ for all }f\in
C_b^\infty(\R^d).\endaligned
\end{equation}
In order to show that \eqref{sec1.1.9} is qualitatively
sharp, we will study the concentration of measure for the Poincar\'{e}
inequality \eqref{sec1.1.10} of $D_{\alpha,V,\delta}$ with $\delta>0$.

\begin{proposition}\label{prop1.3} Let $\delta>0$ and $0<\alpha<2$, and let $\mu_V$ be a probability measure defined by \eqref{sec1.1.6}. Suppose that there is a
constant $C_2>0$ such that
the Poincar\'{e} inequality \eqref{sec1.1.10} holds for such $\mu_V$ and $D_{\alpha,V,\delta}$. Then, there is a constant
$\lambda_0>0$ such that $$\int e^{\lambda_0 |x|}\mu_V(dx)<\infty.$$
\end{proposition}


\subsection{Weighted Poincar\'{e} Inequalities for $D_{\alpha,V}$: Completeness of the Work in \cite{WW}}
Let $D_{\alpha,V}$ be the bilinear form defined by (\ref{sec1.1.4}). We have the following result.
\begin{theorem}\label{thm1.4} 
Let $\alpha\in (0,2)$. If for some constant
$\alpha_0\in (0,\alpha/2)$,
\begin{equation}\label{thm1.4.1}
\limsup_{|x|\to\infty}\bigg[
\bigg(\sup_{|z|\geqslant |x|}e^{-V(z)}\bigg)|x|^{d+\alpha-\alpha_0}\bigg]=0,
\end{equation} then there exists a constant $C_2>0$ such that the following
weighted Poincar\'{e} inequality
\begin{equation}\label{thm1.4.2}\aligned
\int \big(f (x)-\mu_V(f)\big)^2
\frac{e^{V(x)}}{(1+|x|)^{d+\alpha}}\,\mu_V(dx)\le C_2 D_{\alpha,V}(f,f)\endaligned
\end{equation}
holds for all $f\in C_b^\infty(\R^d)$.

The weighted function in the weighted Poincar\'{e} inequality \eqref{thm1.4.2} is
$$\omega(x):=\frac{e^{V(x)}}{(1+|x|)^{(d+\alpha)}}.$$
This function is optimal in the sense that, the inequality
\eqref{thm1.4.2} fails if we replace $\omega(x)$ by a positive function $\omega^*(x)$,
which satisfies that
$$\liminf_{|x|\to\infty}\frac{\omega^*(x)}{\omega(x)}=\infty.$$
\end{theorem}

Theorem \ref{thm1.4} can be seen as a complement of \cite[Theorem
1.1]{WW}, where explicit criteria are presented for fractional
Dirichlet form $D_{\alpha,V}$ to satisfy Poincar\'{e}, super
Poincar\'{e} and weak Poincar\'{e} inequalities.

\bigskip

We first mention
that the weighted Poincar\'{e} inequality \eqref{thm1.4.2} can be
satisfied for some probability measures, which do not fulfill the
true Poincar\'{e} inequality.

\begin{example}\label{exm1.1.5} For $\varepsilon>0$, let $\mu_{\varepsilon}(dx)={C_{\varepsilon}}{(1+|x|^2)^{-(d+\varepsilon)/2}}\,dx$ be a probability measure, where
$C_{\varepsilon}$ is a normalizing constant. According to Theorem
\ref{thm1.4}, if $\varepsilon>\alpha/2$, there exists a constant
$c_3>0$ such that the following weighted Poincar\'{e} inequality
$$\aligned \int \big(f (x)-\mu_\varepsilon(f)\big)^2
\frac{1}{(1+|x|)^{\alpha-\varepsilon}}&\,\mu_\varepsilon(dx)\le c_3\iint \frac{(f(y)-f(x))^2}{|y-x|^{d+\alpha}}\,dy
\,\mu_\varepsilon(dx)\endaligned$$ holds for all $f\in
C_b^\infty(\R^d)$. However, by \cite[Corollary 1.2 (1)]{WW}, we know
that the following Poincar\'{e} inequality $$\aligned \int
\!\!\big(f (x)-
\mu_\varepsilon(f)\big)^2&\,\mu_\varepsilon(dx)\!\le
c_4\!\!\iint \frac{(f(y)-f(x))^2}{|y-x|^{d+\alpha}}\,dy
\,\mu_\varepsilon(dx),\quad f\in C_b^\infty(\R^d) \endaligned$$ does
not hold for any $\varepsilon\in(\alpha/2,\alpha).$ \end{example}

\medskip

The next result shows that the weighted Poincar\'{e} inequality for $D_{\alpha,V}$ with continuous
weighted function, which tends to infinite when $|x|$ tends to infinite, indeed implies the super Poincar\'{e} inequality.
For any $r>0$, define \begin{equation*}
h(r):=\inf_{|x| \le r}e^{V(x)},\quad H(r):=\sup_{|x| \le r}e^{V(x)}.
\end{equation*}
\begin{proposition}\label{prop1.6.0}
Let $\mu_V$ be a probability measure given by \eqref{sec1.1.6}, and $\omega$ be a positive continuous function on $\R^d$ such that $\lim_{|x|\rightarrow \infty} \omega(x)=\infty$. Suppose that there is a constant $C_0>0$ such that
the following weighted Poincar\'{e} inequality holds
\begin{equation}\label{e1}
\begin{split}
\int \big(f(x)-\mu_V(f)\big)^2 \omega(x) \,\mu_V(dx)\leqslant C_0 D_{\alpha,V}(f,f),\quad f \in C_b^{\infty}(\R^d).
\end{split}
\end{equation}
Then the following super Poincar\'{e} inequality
\begin{equation}\label{e2}
\mu_V(f^2)\le r D_{\alpha,V}(f,f)+\beta(r)\mu_V(|f|)^2,\quad r>0,\,\ f\in C_b^{\infty}(\R^d)
\end{equation}
holds with
$$\beta(r)=\inf\bigg\{{C_1H(t)^{2+{d}/{\alpha}}}
{h(t)^{-1-{d}/{\alpha}}}\big(1+s^{-{d}/{\alpha}}\big):\,\,\frac{2C_0}{\inf\limits_{|x|\ge t}\omega(x)}+s\le r, t>1, s>0  \bigg\}.$$
In particular, there are $r_0>0$ small enough and a constant $C_2>0$ such that for all $0<r\le r_0$,
\begin{equation*}
\beta(r)\le C_2\Big(1+r^{-{d}/{\alpha}}\big(h\circ \kappa(4C_0r^{-1})\big)^{-1-{d}/{\alpha}}
\big(H\circ \kappa(4C_0r^{-1})\big)^{2+{d}/{\alpha}}\Big),
\end{equation*} where
\begin{equation*}
\kappa(r):=\inf\big\{s>0:\ \inf_{|x|\ge s}\omega(x) \ge r\big\}.
\end{equation*}
\end{proposition}

Let $V(x)=\frac{d+\varepsilon}{2}\log(1+|x|^2)$ with $\varepsilon>\alpha$, and  $$\mu_V(dx)=C_V e^{-V(x)}\,dx=C_V{(1+|x|^2)^{-(d+\varepsilon)/2}}\ dx$$ be the corresponding probability measure. According to Proposition \ref{prop1.6.0} and the weighted Poincar\'{e} inequality obtained in Example \ref{exm1.1.5}, we know that the super Poincar\'{e} inequality \eqref{e2} holds for
such $\mu_V$ and $D_{\alpha,V}$ with $$\beta(r)=c\Big(1+r^{-\frac{d}{\alpha}-\frac{(d+\varepsilon)(2\alpha+d)}{\alpha(\varepsilon-\alpha)}}\Big).$$
This estimate for the rate function $\beta$ is exactly the same as that in \cite[Corollary 1.2 (2)]{WW}, which indicates that the estimate above is optimal. However, due to the non-local property, we do not know whether the super Poincar\'{e} inequality
implies the weighted Poincar\'{e} inequality for the Dirichlet form $D_{\alpha,V}$, although it is true for local Dirichlet forms, e.g.\ see \cite{W3}.

\medskip

There exist a lot of
works for weighted Poincar\'{e} type inequalities for local
Dirichlet forms, e.g.\ see \cite{BL, CGW}. The difference between
the main results in those cited papers and Theorem \ref{thm1.4} is
that, the weighted function of weighted Poincar\'{e} inequalities in
\cite[Theorem 3.1]{BL} and \cite[Theorem 2.1]{CGW} is inside the
associated Dirichlet form, but the weighted function of the
inequality \eqref{thm1.4.2} here appears in the variation term (i.e.\ the left
hand side of the inequality). The following proposition shows that the
weighted Poincar\'{e} inequality \eqref{thm1.4.2} implies more
information, which may indicate that weighted Poincar\'{e}
inequalities of the form \eqref{thm1.4.2} are more suitable to
study for non-local Dirichlet forms.

\begin{proposition}\label{prop1.6} Let the function $V$ satisfying
\begin{equation}\label{prop1.6.1}
\liminf_{|x|\rightarrow \infty}\frac{e^{V(x)}}{|x|^{d+\alpha}}>0,
\end{equation}
and let $\omega:\R^d\to \R_+$ be a continuous and positive function.
Then there exists a constant $C_2(\omega)>0$ such that the following
weighted Poincar\'{e} inequality
\begin{equation}\label{prop1.6.2}\aligned
\int \big(f (x)-\mu_V(f)\big)^2& \frac{e^{V(x)}}{(1+|x|)^{d+\alpha}}
\mu_V(dx)\\
&\le C_2(\omega)\int \omega(x) \int
\frac{(f(y)-f(x))^2}{|y-x|^{d+\alpha}}\,dy \,\mu_V(dx)\endaligned
\end{equation}
holds for all $f\in C_b^\infty(\R^d)$.
 \end{proposition}

Under (\ref{prop1.6.1}), the function $V$ satisfies \eqref{thm1.4.1}, which implies that the inequality
(\ref{thm1.4.2}) holds. According to (\ref{prop1.6.2}), in this
situation we can improve the inequality (\ref{thm1.4.2}) by adding a
weighted function $\omega$, which may tend to $0$ in any rate as
$|x| \rightarrow \infty$, inside the non-local Dirichlet form
$D_{\alpha,V}$.{\footnote{To deduce \eqref{prop1.6.2} from
\eqref{thm1.4.2}, one may take
$C_2(\omega)=\sup_{x\in\R^d}\omega(x)^{-1}.$ However, as mentioned
above the weighted function $\omega$ may tend to $0$ as $|x|
\rightarrow \infty$, and so in this case $C_2(\omega)$ is infinite, which does not work.
We will see below that the proof of Proposition \ref{prop1.6} is not
trivial.}} However, in the context of local Dirichlet forms, to
obtain such weighted Poincar\'{e} inequality we need to put some
restrictive conditions on the rate of decay for the weighted
function $\omega$.{\footnote{One can easily check this point by
using the criteria about Poincar\'{e} inequalities for one
dimensional diffusion processes, e.g.\ see \cite[Table 1.4, Page
15]{CMF}.}} Roughly speaking, the difference is as follows: applying
$f \in C_b^{\infty}(\R^d)$ with support contained in the set $\{x\in
\R^d: |x|>r\}$ into the weighted local Dirichlet form
$$\tilde D_{\omega}(f,f):=\int \omega(x)|\nabla f(x)|^2\, \mu_V(dx),$$ we find that $\tilde D_{\omega}(f,f)$ only depends on the value of
$\omega$ in domain $\{x\in \R^d: |x|>r\}$; while for the weighted non-local Dirichlet form $$D_{\omega}(f,f):=\int \omega(x) \int
\frac{(f(y)-f(x))^2}{|y-x|^{d+\alpha}}\,dy \,\mu_V(dx),$$
$D_{\omega}(f,f)$ depends on the value of
$\omega$ in  $\R^d$, not only in $\{x\in \R^d: |x|>r\}$.

\bigskip

The remaining part of this paper is organized as follows. In the
next section we recall results on non-local Dirichlet forms and
their generator, which apply all the examples to be studied in our paper. Section \ref{sec3} is devoted to
general theory
 on the existence of weighted Poincar\'{e} inequalities for non-local Dirichlet forms via
Lyapunov-type conditions of the associated truncated Dirichlet
forms, which shall be interesting of itself. In Section \ref{sec4},
we establish efficient Lyapunov conditions for the truncated Dirichlet
form associated with original Dirichlet form. This, along with the results in
Sections \ref{sec2} and \ref{sec3}, gives us weighted Poincar\'{e}
inequalities for general non-local Dirichlet forms (see Theorem
\ref{th4.1}), which immediately yield Theorems \ref{thm1.1} and
\ref{thm1.4}. The proofs of Proposition \ref{prop1.3}, Proposition \ref{prop1.6.0} and Proposition
\ref{prop1.6} are also included here. In Section 5, we will state that
our approach to Theorem \ref{th4.1} also yields the criterion for
weighted Poincar\'{e} inequalities for Dirichlet forms associated
with symmetric Markov processes under Girsanov transform of pure
pump type. We also consider the corresponding concentration of
measure, which indicates that the inequalities we derived above are
optimal in some sense.

\section{Characterization of Operators Associated with Non-local Dirichlet Forms}\label{sec2}
\subsection{Non-local Dirichlet Forms in Terms of Generators}
Let $C_c^\infty(\R^d)$ be the set of smooth functions with compact
support on $\R^d$. Let $V$ be a locally bounded measurable function
on $\R^d$ such that $\int \,e^{-V(x)}\,dx<\infty,$ and $j$ be a
measurable function on $\R^{2d}\setminus\{(x,y)\in \R^{2d};\ x=y\}$
such that $j(x,y)\geqslant 0$ and $j(x,y)=j(y,x)$. Let
$\mu_V(dx)=C_V e^{-V(x)}\,dx$ be a probability measure on $(\R^d,
\mathscr{B}(\R^d))$ with a normalizing constant
$$ C_V:=\frac{1}{\int e^{-V(x)}\,dx}. $$

Consider
$$
\aligned D_{j,V}(f,g):=&\frac{1}{2}\iint
{\big(f(y)-f(x)\big)\big(g(y)-g(x)\big)}j(x,y)\,\mu_V(dy)\,\mu_V(dx),\\
\mathscr{D}(D_{j,V}):=&\bigg\{f\in L^2(\mu_V):
D_{j,V}(f,f)<\infty\bigg\}.
\endaligned
$$
Note that the kernel $j(x,y)$ is only defined on the set $\{(x,y) \in \R^{2d}: x
\neq y\}$. Since $\{(x,y) \in \R^{2d}: x =y\}$ is a zero-measure set
under $\mu_V(dx)\,\mu_V(dy)$, we can still write the integral domain
as $\R^d\times \R^d$ in the expression above for $D_{j,V}$.

Suppose that for any $\varepsilon>0$, the function $$x \mapsto \int_{\{|x-y|>\varepsilon\}}j(x,y)\,\mu_V(dy)$$ is locally
 integrable with respect to $\mu_V(dx)$. Then, according to the proof of \cite[Example 1.2.4]{FOT}, $(D_{j,V}, \mathscr{D}(D_{j,V}))$ is
a symmetric Dirichlet form on $L^2(\mu_V)$ in the wide sense;
namely, the set $\mathscr{D}(D_{j,V})$ is not necessarily dense in
$L^2(\mu_V)$, c.f.\ see \cite[Chapter 1.3]{FOT}.

\begin{theorem}\label{thm2.1} The following three statements are satisfied.

\begin{itemize}
\item[(1)] If
\begin{equation}\label{thm2.1.1}I_1(x):=\int\big(1\wedge
|x-y|^2\big)j(x,y)\,\mu_V(dy)\in
L_{\textrm{loc}}^1(\mu_V),\end{equation} then
$C_c^\infty(\R^d)\subset \mathscr{D}(D_{j,V})$; if
\begin{equation}\label{thm2.1.2}I_1(x):=\int\big(1\wedge
|x-y|^2\big)j(x,y)\,\mu_V(dy)\in L^1(\mu_V),\end{equation} then
$C_b^\infty(\R^d)\subset \mathscr{D}(D_{j,V})$. In particular, in
both cases $(D_{j,V}, \mathscr{D}(D_{j,V}))$ is a Dirichlet form on
$L^2(\mu_V)$.

\item[(2)] For any $x\in\R^d$, set
 $$I_2(x):=\int_{\{|z|\le 1\}}|z|\Big| j(x,x+z)e^{-V(x+z)}-j(x,x-z)e^{-V(x-z)}\Big|\,dz.$$ Suppose that 
\begin{equation}\label{thm2.1.3}\textrm{ the function }x\mapsto I_i(x) \ \textrm{ is  locally  bounded for }i=1,\,2.\end{equation}
Then, for any $f$, $g\in C_c^\infty(\R^d)$, $$D_{j,V}(f,g)=-\int
gL_{j,V}f\,d\mu_V,$$ where
\begin{equation}\label{thm2.1.4}
\aligned &L_{j,V}f(x)\\
&=C_V\bigg[\int \Big (f(x+z)-f(x)-\nabla f(x)\cdot
z\I_{\{|z|\le 1\}}\Big)j(x,x+z)e^{-V(x+z)}\,dz\\
&\qquad\,\,\,+\frac{1}{2}\nabla f(x) \cdot \int_{\{|z|\le 1\}}
z\Big(j(x,x+z)e^{-V(x+z)}-j(x,x-z)e^{-V(x-z)}\Big)\,dz\bigg].
\endaligned
\end{equation}

\item[(3)] Suppose that \eqref{thm2.1.3} holds, and
there is a constant $r_0>0$ such that for each $r\ge r_0$,
\begin{equation}\label{thm2.1.5}
I_{3,r}(x):=\I_{(B(0,2r))^c}(x)\int_{\{|z+x|\leqslant r\}} j(x,x+z)
e^{-V(x+z)}\,dz \in L^2(\mu_V).
\end{equation}
Then for each $f\in C_c^{\infty}(\R^d)$, $L_{j,V}f \in L^2(\mu_V)$.

\end{itemize}
\end{theorem}

\begin{proof} (a) For any $f\in C_c^\infty(\R^d)$, choose
$r_2$, $r_1$ large enough such that $r_2>r_1+1$ and
$\text{supp}(f)\subseteq B(0,r_1).$ Set
$$c_0(f):=\text{max}\Big\{\sup_{x \in \R^d}|\nabla f(x)|^2,\
4\sup_{x \in \R^d}|f(x)|^2\Big\}.$$ Then,
by the symmetric property that $j(x,y)=j(y,x)$,
\begin{equation}\label{c1}
\aligned D_{j,V}&(f,f)\\
=&\iint_{B(0,r_2)\times
B(0,r_2)}(f(x)-f(y))^2j(x,y)\,\mu_V(dy)\,\mu_V(dx)\\
&+\iint_{B(0,r_2)\times
B(0,r_2)^c}f(x)^2j(x,y)\,\mu_V(dy)\,\mu_V(dx)\\
&+\iint_{B(0,r_2)^c\times
B(0,r_2)}f(y)^2j(x,y)\,\mu_V(dy)\,\mu_V(dx)\\
=&\iint_{B(0,r_2)\times
B(0,r_2)}(f(x)-f(y))^2j(x,y)\,\mu_V(dy)\,\mu_V(dx)\\
&+2\iint_{B(0,r_2)\times
B(0,r_2)^c}f(x)^2j(x,y)\,\mu_V(dy)\,\mu_V(dx)\\
\le &\sup_{x \in \R^d}|\nabla f(x)|^2\iint_{B(0,r_2)\times
B(0,r_2)}|x-y|^2j(x,y)\,\mu_V(dy)\,\mu_V(dx)\\
&+4\sup_{x \in \R^d}|f(x)|^2\iint_{B(0,r_1)\times
B(0,r_2)^c}j(x,y)\,\mu_V(dy)\,\mu_V(dx)\\
\le&c_0(f)\bigg[\int_{B(0,r_2)}\int_{\{|x-y|\le
2r_2\}}|x-y|^2j(x,y)\,\mu_V(dy)\,\mu_V(dx)\\
&\qquad\qquad+\int_{B(0,r_1)}\int_{\{|x-y|\ge r_2-r_1\}}j(x,y)\,\mu_V(dy)\,\mu_V(dx)\bigg]\\
\le&c_0(f)\bigg[\int_{B(0,r_2)}\int
\big(|x-y|^2\wedge (4r_2^2)\big)j(x,y)\,\mu_V(dy)\,\mu_V(dx)\\
&\qquad\qquad+\int_{B(0,r_1)}\int \big(|x-y|^2\wedge (r_2-r_1)^2\big)j(x,y)\,\mu_V(dy)\,\mu_V(dx)\bigg]\\
\le &2c_0(f)\int _{B(0,r_2)}\int \big(|x-y|^2\wedge
(4r_2^2)\big)j(x,y)\,\mu_V(dy)\,\mu_V(dx)\\
\le &8r_2^2c_0(f)\int _{B(0,r_2)}\int \big(|x-y|^2\wedge
1\big)j(x,y)\,\mu_V(dy)\,\mu_V(dx).
\endaligned
\end{equation} This, along with \eqref{thm2.1.1}, yields the first conclusion of part
(1).

For each $f\in C_b^{\infty}(\R^d)$, we still set
$$c_0(f):=\text{max}\Big\{\sup_{x \in \R^d}|\nabla f(x)|^2,\
4\sup_{x \in \R^d}|f(x)|^2\Big\}.$$
Then, by the mean value theorem, for any $x$, $y\in\R^d$,
$$|f(x)-f(y)|^2\leqslant c_0(f)\big(1\wedge |x-y|^2\big).$$ Hence,
$D_{j,V}(f,f)<\infty$, if (\ref{thm2.1.2}) holds. This proves the
second desired assertion of part (1).

(b) The proof of part (2) essentially follows from that of
\cite[Theorem 1.2]{WJ2}. For the sake of completeness, here we
present the proof in a different and simple way. We first note that
under \eqref{thm2.1.3}, \eqref{thm2.1.1} is satisfied, and so
$C_c^\infty(\R^d)\subset \mathscr{D}(D_{j,V})$. For each $f\in
C_c^{\infty}(\R^d)$, by \eqref{thm2.1.4} and the mean value theorem, we
have
\begin{equation*}
\begin{split}
&L_{j,V}f(x)\leqslant c_1\big(I_1(x)+I_2(x)\big)
\end{split}
\end{equation*}
for some constant $c_1:=c_1(f)>0$. Hence, (\ref{thm2.1.3}) implies that $L_{j,V}f$
is well defined and locally bounded.

Next, for each $\varepsilon\in(0,1)$ and $f\in C_c^\infty(\R^d)$, define
\begin{equation*}
\aligned L_{j,V,\varepsilon}&f(x)\\
=&C_V\bigg[\int_{\{|z|>\varepsilon\}} \Big (f(x+z)-f(x)-\nabla f(x)\cdot
z\I_{\{|z|\le 1\}}\Big)j(x,x+z)e^{-V(x+z)}\,dz\\
&\qquad\,\,\,+\frac{1}{2}\nabla f(x) \cdot \int_{\{\varepsilon<|z|\le 1\}}
z\Big(j(x,x+z)e^{-V(x+z)}-j(x,x-z)e^{-V(x-z)}\Big)\,dz\bigg]\\
=&C_V \int_{\{|z|>\varepsilon\}} \big (f(x+z)-f(x)\big)j(x,x+z)e^{-V(x+z)}\,dz\\
&\qquad-\frac{C_V}{2}\nabla f(x) \cdot \int_{\{\varepsilon<|z|\le 1\}}
z\Big(j(x,x+z)e^{-V(x+z)}+j(x,x-z)e^{-V(x-z)}\Big)\,dz\\
=:&L_{1,\varepsilon}f(x)+L_{2,\varepsilon}f(x).
\endaligned
\end{equation*}
Since for any $x\in\R^d$ and $z\in\R^d$ with $|z|\ge \varepsilon$,
$$j(x,x+z)\leqslant \Big(\frac{|z|^2}{\varepsilon^2}\wedge
1\Big)j(x,x+z),$$ the condition $I_1(x)$ is locally bounded implies
that $L_{i,\varepsilon}f$, for $i=1,$ $2$ and any
$\varepsilon\in(0,1)$, are well defined and locally bounded. By the
change of variable from $z$ to $-z$ and the symmetric property of
$j(x,y)$, we have $L_{2,\varepsilon}f(x)=0$ for all $x\in \R^d$.
That is. $L_{j,V,\varepsilon}=L_{1,\varepsilon}.$ Hence, for each $f$, $g\in C_c^{\infty}(\R^d)$,
\begin{equation*}
\begin{split}
-\int L_{j,V,\varepsilon}f(x)g(x)\,\mu_V(dx)&=-
\int L_{1,\varepsilon}f(x)g(x)\,\mu_V(dx)\\
&=-\iint_{\{|x-y|>\varepsilon\}}\big(f(y)-f(x)\big)g(x)j(x,y)\,\mu_V(dy)\,\mu_V(dx).
\end{split}
\end{equation*}
Changing the position of $x$ and $y$, it holds that
\begin{equation*}
-\int L_{j,V,\varepsilon}f(x)g(x)\,\mu_V(dx)=
\iint_{\{|x-y|>\varepsilon\}}\big(f(y)-f(x)\big)g(y)j(x,y)\,\mu_V(dy)\,\mu_V(dx),
\end{equation*}
Therefore, combining two equalities above, we have
\begin{equation}\label{sec2.1.20}
\begin{split}
-\int& L_{j,V,\varepsilon}f(x)g(x)\,\mu_V(dx)\\
&=\frac{1}{2}
\iint_{\{|x-y|>\varepsilon\}}\big(f(y)-f(x)\big)\big(g(y)-g(x)\big)j(x,y)\,\mu_V(dy)\,\mu_V(dx).
\end{split}
\end{equation}

By the mean value theorem and (\ref{thm2.1.3}), in the support of $g$
the function $L_{j,V,\varepsilon}f(x)$ is uniformly bounded for any
$\varepsilon\in (0,1)$. Thus, the dominated convergence theorem
yields that
\begin{equation*}
\lim_{\varepsilon \rightarrow 0}\int L_{j,V,\varepsilon}f(x)g(x)\,\mu_V(dx)=
\int L_{j,V}f(x)g(x)\,\mu_V(dx)
\end{equation*}
On the other hand, according to estimates in (\ref{c1}) and also the dominated convergence theorem,
\begin{equation*}
\lim_{\varepsilon \rightarrow 0}\frac{1}{2}
\iint_{\{|x-y|>\varepsilon\}}\big(f(y)-f(x)\big)\big(g(y)-g(x)\big)j(x,y)\,\mu_V(dy)\,\mu_V(dx)
=D_{j,V}(f,g).
\end{equation*}
Then, letting $\varepsilon \rightarrow 0$ in (\ref{sec2.1.20}), we prove the conclusion of
part (2).

(c) For the part (3), the proof is almost the same as that of
\cite[Lemma 2.1]{WJ2}, but we can get an improvement of the
conclusion by a minor modification. In fact, for each $f\in
C_c^{\infty}(\R^d)$,  there is a constant $r>r_0$ such that
$\text{supp}(f)\subseteq B(0,r)$. As mentioned in the proof of part
(2) above, under (\ref{thm2.1.3}) the function $L_{j,V}f$ is locally bounded, and
so $\|\I_{B(0,2r)}L_{j,V}f\|_{L^2(\mu_V)}<\infty$. On the other
hand, since for any $|x|>r$, $f(x)=0$ and $\nabla f(x)=0$, it
follows from \eqref{thm2.1.4} that
\begin{equation*}
\begin{split}
\Big|\I_{(B(0,2r))^c}(x)&L_{j,V}f(x)\Big|\\
&=
\Big|C_V\I_{(B(0,2r))^c}(x)\int f(x+z)j(x,x+z)e^{-V(x+z)}\,dz\Big|\\
&\leqslant C_V\|f\|_{\infty}\I_{(B(0,2r))^c}(x)
\int \I_{B(0,r)}(x+z)j(x,x+z)e^{-V(x+z)}\,dz\\
&= C_V\|f\|_{\infty}\I_{(B(0,2r))^c}(x) \int_{\{ |z+x|\leqslant
r\}}j(x,x+z)e^{-V(x+z)}\,dz,
\end{split}
\end{equation*}
Then, by \eqref{thm2.1.5}, we get that for any $r>r_0$,
$\|\I_{(B(0,2r))^c}L_{j,V}f\|_{L^2(\mu_V)}<\infty$. This completes
the proof.
\end{proof}

\medskip

Under \eqref{thm2.1.1}, let $(D_{j,V}, \mathscr{E}(D_{j,V}))$ be
the closure of $(D_{j,V}, C_c^\infty(\R^d))$ under norm
$\|.\|_{D_{j,V},1} $ on $L^2(\mu_V)$, where
$\|f\|_{D_{j,V},1}:=\big(\|f\|_{L^2(\mu_V)}^2+D_{j,V}(f,f)\big)^{1/2}$
for $f\in C_c^{\infty}(\R^d)$. Then, the bilinear form $(D_{j,V},
\mathscr{E}(D_{j,V}))$ becomes a regular Dirichlet form on
$L^2(\mu_V)$. It holds that $\mathscr{E}(D_{j,V})\subseteq
\mathscr{D}(D_{j,V})$. However, those two domains may be different,
and the Dirichlet form $(D_{j,V}, \mathscr{D}(D_{j,V}))$ may not be
regular in generally. On the other hand, we note that under \eqref{thm2.1.3}
the operator $L_{j,V}$ does not necessarily map $C_c^\infty(\R^d)$
into $L^2(\mu_V)$, though it is well defined in the sense of
pointwise on $C_c^\infty(\R^d)$. If moreover (\ref{thm2.1.5}) holds,
then the Friedrich extension of $(L_{j,V},C_c^{\infty}(\R^d))$ is a
self-joint operator, which is the infinitesimal generator of the
Dirichelt form $(D_{j,V}, \mathscr{E}(D_{j,V}))$.

\subsection{Examples}
In this part, we will present several examples as an application of
Theorem \ref{thm2.1}. In all the examples, let $V$ be a locally
bounded function on $\R^d$ such that $\int e^{-V(x)}\,dx<\infty$ and
$ e^{-V}$ is bounded in $\R^d$. Let $\mu_V(dx)=C_V e^{-V(x)}\,dx$ be
a probability measure on $(\R^d, \mathscr{B}(\R^d))$.

\begin{example}\label{exm2.1}  Let
$\rho$ be a positive measurable function on $\R_+:=(0,\infty)$ such that
\begin{equation}\label{exm2.1.1}
\int_{(0,\infty)} \rho(r)\big(1\wedge r^2\big)r^{d-1}\,dr<\infty.
\end{equation} Consider the following form
\begin{equation}\label{exm2.1.2}\aligned
D_{\rho,V}(f,g):=&\frac{1}{2}\iint
{\big(f(y)-f(x)\big)\big(g(y)-g(x)\big)}\rho(|x-y|)\,dy\,\mu_V(dx),\\
\mathscr{D}(D_{\rho,V}):=&\bigg\{f\in L^2(\mu_V): D_{\rho,
V}(f,f)<\infty\bigg\}.\endaligned
\end{equation}
Then, $(D_{\rho,V}, \mathscr{D}(D_{\rho,V}))$ is a symmetric
Dirichlet form on $L^2(\mu_V)$ such that $C_b^\infty(\R^d)\subset
\mathscr{D}(D_{\rho,V})$. Moreover, if $e^{-V}\in C_b^1(\R^d)$, then for any $f$, $g\in C_c^\infty(\R^d)$,
$$D_{\rho,V}(f,g)=-\int gL_{\rho,V}f\,d\mu_V,$$ where
\begin{equation}\label{exm2.1.3}
\begin{split}
&L_{\rho,V}f(x)\\&=\frac{1}{2}\bigg[\int \Big (f(x+z)-f(x)-\nabla
f(x)\cdot
z\I_{\{|z|\le 1\}}\Big)\rho(|z|) \big(e^{V(x)-V(x+z)}+1\big)\,dz\\
&\qquad\quad+\nabla f(x) \cdot \int_{\{|z|\leqslant 1\}} z
\rho(|z|)\big(e^{V(x)-V(x+z)}-1\big)\,dz\bigg].
\end{split}
\end{equation}
Additionally, if for $r$ big enough,
\begin{equation}\label{exm2.1.4} \int_{\{|x|\ge
2r\}}\Big(\sup_{z\in\R^d:|z|\ge
|x|-r}\rho(|z|)\Big)^2e^{V(x)}\,dx<\infty,
\end{equation}
then $L_{\rho,V}$ maps $C_c^\infty(\R^d)$ into $L^2(\mu_V)$.
\end{example}

\begin{proof}
By changing the position of $x$, $y$ and the symmetric property of $D_{\rho,V}$, it
is easy to see that for any $f$, $g\in C_c^\infty(\R^d)$,
$$
\aligned D_{\rho,V}&(f,g)=\frac{1}{2}\iint{\big(f(y)\!-f(x)\big)\big(g(y)-g(x)\big)}j(x,y)\,\mu_V(dy)\,\mu_V(dx),
\endaligned
$$
where
$$j(x,y)=\frac{1}{2C_V}\rho(|x-y|)\big(e^{V(x)}+e^{V(y)}\big).$$

First, for all $x\in\R^d$,
\begin{equation}\label{exm2.1.5}
I_1(x)=\frac{1}{2}\int \big(1\wedge
|x-y|^2\big)\rho(|x-y|)\big(e^{V(x)-V(y)}+1\big)\,dy,
\end{equation}
so
\begin{equation*}
\begin{split}
\int I_1(x)\, \mu_V(dx)=&
\frac{C_V}{2}\bigg[\iint\big(1\wedge |x-y|^2\big)\rho(|x-y|)e^{-V(x)}\,dy\,dx\\
&\qquad+\iint\big(1\wedge |x-y|^2\big)\rho(|x-y|)e^{-V(y)}\,dy\,dx\bigg]\\
=&C_V \iint\big(1\wedge |x-y|^2\big)\rho(|x-y|)\,dye^{-V(x)}\,dx
\end{split}
\end{equation*}
Then, according to (\ref{exm2.1.1}) and $e^{-V}\in L^1(dx)$, $$\int
I_1(x)\, \mu_V(dx)<\infty,$$ i.e.\ (\ref{thm2.1.2}) is true.
According to Theorem \ref{thm2.1} (1), $(D_{\rho,V},
\mathscr{D}(D_{\rho,V}))$ is a symmetric Dirichlet form on
$L^2(\mu_V)$ such that $C_b^\infty(\R^d)\subset
\mathscr{D}(D_{\rho,V})$.

Second, since $ e^{-V}$ is bounded and $V$ is locally bounded, by (\ref{exm2.1.5}) and
(\ref{exm2.1.1}), it is easy to check $I_1(x)$ is locally bounded. On the
other hand,
\begin{equation*}
I_2(x)= \frac{1}{2C_V}\int_{\{|z|\le 1\}} |z|
\rho(|z|)\Big|e^{V(x)-V(x+z)}-e^{V(x)-V(x-z)}\Big|\,dz,
\end{equation*}
Due to the fact $ e^{-V}\in C_b^1(\R^d)$, it follows from
(\ref{exm2.1.1}) and the mean value theorem that $I_2(x)$ is locally
bounded. Thus, (\ref{thm2.1.3}) holds. Therefore, according to Theorem
\ref{thm2.1} (2), we know that for any $f$, $g\in C_c^\infty(\R^d)$,
$$D_{\rho,V}(f,g)=-\int gL_{\rho,V}f\,d\mu_V,$$ where
\begin{equation*}
\begin{split}
L_{\rho,V}f(x)=&\frac{1}{2}\bigg[\int \Big (f(x+z)-f(x)-\nabla
f(x)\cdot
z\I_{\{|z|\le 1\}}\Big)\rho(|z|) \big(e^{V(x)-V(x+z)}+1\big)\,dz\\
&\qquad+\frac{1}{2}\nabla f(x) \cdot \int_{\{|z|\leqslant 1\}} z
\rho(|z|)\Big(e^{V(x)-V(x+z)}-e^{V(x)-V(x-z)}\Big)\,dz\bigg].
\end{split}
\end{equation*}
Combining it with the fact
$$ \aligned \int_{\{|z|\leqslant 1\}} z
\rho(|z|)&\Big(e^{V(x)-V(x+z)}-e^{V(x)-V(x-z)}\Big)\,dz\\
&=\int_{\{|z|\leqslant 1\}} z
\rho(|z|)\Big(\big(e^{V(x)-V(x+z)}-1\big)-\big(e^{V(x)-V(x-z)}-1\big)\Big)\,dz\\
&=2\int_{\{|z|\le 1\}} z
\rho(|z|)\big(e^{V(x)-V(x+z)}-1\big)\,dz\endaligned
$$
yields the required assertion (\ref{exm2.1.3}).

\medskip

Finally, for $r$ large enough, we have
\begin{equation*}
\begin{split}
I_{3,r}(x)=\frac{1}{2C_V}\I_{(B(0,2r))^c}(x)\int_{\{|z+x| \leqslant
r\}} \rho(|z|)\big( 1+e^{V(x)-V(x+z)}\big)\,dz,
\end{split}
\end{equation*}
According to \eqref{exm2.1.1} and the facts that $e^{-V}$ is bounded and for any
$|x|\ge 2r$ and $|x+z|\le r$, $$|z|\ge |x|-|x+z|\ge |x|-r\ge r,$$ we get that
for $r>0$ large enough, there exist $c_i:=c_i(r)>0$ $(i=1,2,3)$ such that
\begin{equation*}\aligned I_{3,r}(x)&\leqslant
c_1\I_{(B(0,2r))^c}(x)\bigg[\int_{\{|z|\ge r\}}\rho(|z|)\,dz\\
&\qquad\qquad \qquad\qquad\qquad\qquad+
e^{V(x)}\Big(\int_{\{|x+z|\le r\}}e^{-V(x+z)}\,dz\Big)
\Big(\sup_{|x+z|\le r}\rho(|z|)\Big)\bigg]\\
&\leqslant c_2\I_{(B(0,2r))^c}(x)\bigg[1+
e^{V(x)}\Big(\int_{\{|z|\le r\}}e^{-V(z)}\,dz\Big) \Big(\sup_{|z|\ge
|x|- r}\rho(|z|)\Big)\bigg]\\
&\leqslant c_3\I_{(B(0,2r))^c}(x)\bigg[1+
e^{V(x)}\Big(\sup_{|z|\ge
|x|- r}\rho(|z|)\Big)\bigg],\endaligned
\end{equation*}
which, along with (\ref{exm2.1.4}), yields that $I_{3,r}\in L^2(\mu_V)$.
Hence, by Theorem \ref{thm2.1} (3), we know that $L_{\rho,V}f \in
L^2(\mu_{V})$ for every $f \in C_c^{\infty}(\R^d)$.
\end{proof}

Since $e^{-V}\in L^1(dx)$ and $e^{-V}$ is bounded, $e^{-2V}\in
L^1(dx)$, and hence we can define a probability measure
$$\mu_{2V}(dx):=\frac{1}{\int
e^{-2V(x)}dx}e^{-2V(x)}\,dx=:C_{2V}e^{-2V(x)}dx.$$

\begin{example}\label{exm2.2} Let $\psi$ be a positive
measurable function on $(0,\infty)$ satisfying
\begin{equation}\label{exm2.2.1} \int  \psi(r)\big(1\wedge
r^2\big)r^{d-1}\,dr<\infty.
\end{equation} Consider
the following form
\begin{equation}\label{exm2.2.2}\aligned
D_{\psi,V}(f,g):=&\frac{1}{2}\iint
{\big(f(y)-f(x)\big)\big(g(y)-g(x)\big)}\\
&\qquad\qquad\qquad\qquad\qquad\qquad\times\psi(|x-y|)\,e^{-V(y)}\,dy\,e^{-V(x)}\,dx,\\
\mathscr{D}(D_{\psi,V}):=&\bigg\{f\in L^2(\mu_{2V}): D_{\psi,
V}(f,f)<\infty\bigg\}.\endaligned
\end{equation}
Then, $(D_{\psi,V}, \mathscr{D}(D_{\psi,V}))$ is a symmetric
Dirichlet form on $L^2(\mu_{2V})$ such that $C_b^\infty(\R^d)$
$\subset \mathscr{D}(D_{\psi,V})$. Moreover, if $e^{-V} \in C_b^1(\R^d)$, then for any $f$, $g\in
C_c^\infty(\R^d)$,
$$D_{\psi,V}(f,g)=-\int gL_{\psi,V}f\,d\mu_{2V},$$ where
\begin{equation}\label{exm2.2.3}
\aligned L_{\psi,V}&f(x)\\
&=\frac{1}{C_{2V}}\bigg[\int \Big (f(x+z)-f(x)-\nabla f(x)\cdot
z\I_{\{|z|\leqslant 1\}}\Big)\psi(|z|) e^{V(x)-V(x+z)}\,dz\\
&\qquad\qquad+\nabla f(x) \cdot \int_{\{|z|\leqslant 1\}} z
\psi(|z|)\big(e^{V(x)-V(x+z)}-1\big)\,dz\bigg].
\endaligned
\end{equation}
Additionally, if for $r$ big enough, \begin{equation}\label{exm2.2.4} \int_{\{|x|\ge
2r\}}\Big(\sup_{|z|\ge |x|-r}\psi(|z|)\Big)^2\,dx<\infty.
\end{equation}
then $L_{\psi,V}f \in L^2(\mu_{2V})$ for each $f\in
C_c^{\infty}(\R^d)$.
\end{example}
\begin{proof}
It is easy to see that for any $f$, $g\in C_c^\infty(\R^d)$,
$$
\aligned D_{\psi,V}&(f,g)=\frac{1}{2}\iint_{x\neq
y}{\big(f(x)\!-f(y)\big)\big(g(x)-g(y)\big)}j(x,y)\,\mu_{2V}(dy)\,\mu_{2V}(dx),
\endaligned$$
where
$$j(x,y)=\frac{1}{C_{2V}^2}\psi(|x-y|)\big(e^{V(x)+V(y)}\big).$$
Then,
\begin{equation*}\aligned
I_1(x)=&\frac{1}{C_{2V}}\int \big(1\wedge
|x-y|^2\big)\psi(|x-y|)e^{V(x)-V(y)}dy,\\
I_2(x)=& \frac{1}{C_{2V}^2}\int_{\{|z|\leqslant 1\}} |z|
\psi(|z|)\Big|e^{V(x)-V(x+z)}-e^{V(x)-V(x-z)}\Big|\,dz,\endaligned
\end{equation*}
Since $e^{-V}$ is bounded, $e^{-V} \in L^1(dx)$ and $V$ is locally
bounded, it follows from (\ref{exm2.2.1}) that \eqref{thm2.1.2} holds and
$I_1(x)$ is locally bounded. If $ e^{-V}\in C_b^1(\R^d)$, by
(\ref{exm2.2.1}) and the mean value theorem, we can check that $I_2(x)$ is
also locally bounded. According to Theorem \ref{thm2.1} (1) and (2),
we know that $(D_{\psi,V}, \mathscr{D}(D_{\psi,V}))$ is a symmetric
Dirichlet form on $L^2(\mu_{2V})$ such that $C_b^\infty(\R^d)$
$\subset \mathscr{D}(D_{\psi,V})$, and for any $f$, $g\in
C_c^\infty(\R^d)$,
$$D_{\psi,V}(f,g)=-\int gL_{\psi,V}f\,d\mu_{2V},$$ where
$$ \aligned L_{\psi,V}f(x)=&\frac{1}{C_{2V}}\bigg[\int \Big
(f(x+z)-f(x)-\nabla f(x)\cdot
z\I_{\{|z|\leqslant 1\}}\Big)\psi(|z|) e^{V(x)-V(x+z)}\,dz\\
&\qquad\qquad+\frac{1}{2}\nabla f(x) \cdot \int_{\{|z|\leqslant 1\}}
z \psi(|z|)\Big(e^{V(x)-V(x+z)}-e^{V(x)-V(x-z)}\Big)\,dz\bigg].
\endaligned
$$
Combining this with the fact that
$$ \aligned \int_{\{|z|\leqslant 1\}} &z
\psi(|z|)\Big(e^{V(x)-V(x+z)}-e^{V(x)-V(x-z)}\Big)\,dz\\
&=\int_{\{|z|\leqslant 1\}} z
\psi(|z|)\Big(\big(e^{V(x)-V(x+z)}-1\big)-\big(e^{V(x)-V(x-z)}-1\big)\Big)\,dz\\
&=2\int_{\{|z|\le 1\}} z
\psi(|z|)\big(e^{V(x)-V(x+z)}-1\big)\,dz\endaligned
$$
we get the required expression (\ref{exm2.2.3}).

For $r$ big enough, we have
\begin{equation*}
\begin{split}
I_{3,r}(x)=\frac{1}{C_{2V}^2}\I_{(B(0,2r))^c}(x)\int_{\{|z+x|
\leqslant r\}} \psi(|z|)\big( e^{V(x)-V(x+z)}\big)dz,
\end{split}
\end{equation*}
By the direct computation as that in Example \ref{exm2.1},
\begin{equation*}\aligned I_{3,r}(x)&\leqslant c_1\I_{(B(0,2r))^c}(x)\bigg[
e^{V(x)}\Big(\int_{\{|z|\le r\}}e^{-V(z)}\,dz\Big) \Big(\sup_{|z|\ge
|x|- r}\psi(|z|)\Big)\bigg]\\
&\le c_2\I_{(B(0,2r))^c}(x)\bigg[
e^{V(x)}\Big(\sup_{|z|\ge
|x|- r}\psi(|z|)\Big)\bigg]\endaligned
\end{equation*} holds for some constants $c_i:=c_i(r)$ ($i=1,2$).
This, along with (\ref{exm2.2.4}), implies that $I_{3,r}\in
L^2(\mu_{2V})$. According to Theorem \ref{thm2.1} (3), we know that
$L_{\psi,V}f \in L^2(\mu_{2V})$ for every $f \in C_c^{\infty}(\R^d)$
\end{proof}

In the following example, we consider the same form
$D_{\psi,V}(f,g)$ as that in Example \ref{exm2.2} on the space $L^2(\mu_V)$, not $L^2(\mu_{2V})$.

\begin{example}\label{exm2.3} Let $\psi$ be the same function
as that in Example \ref{exm2.2}. Consider the form $D_{\psi,V}(f,g)$
defined by (\ref{exm2.2.2}), but with the following domain
\begin{equation*}
\mathscr{D}_2(D_{\psi,V}):=\bigg\{f\in L^2(\mu_{V}): D_{\psi,
V}(f,f)<\infty\bigg\}.
\end{equation*}
Then, the bilinear form $(D_{\psi,V}, \mathscr{D}_2(D_{\psi,V}))$ is
a symmetric Dirichlet form on $L^2(\mu_{V})$ such that
$C_b^\infty(\R^d)\subset \mathscr{D}_2(D_{\psi,V})$. Furthermore, if
$e^{-V}\in C_b^1(\R^d)$, then for any
$f$, $g\in C_c^\infty(\R^d)$,
$$D_{\psi,V}(f,g)=-\int gL_{\psi,V,2}f\,d\mu_V,$$ where
\begin{equation*}
\aligned L_{\psi,V,2}f(x)=&{C_{V}}\bigg[\int \Big
(f(x+z)-f(x)-\nabla f(x)\cdot
z\I_{\{|z|\leqslant 1\}}\Big)\psi(|z|) e^{-V(x+z)}\,dz\\
&\qquad+\nabla f(x) \cdot \int_{\{|z|\leqslant 1\}} z
\psi(|z|)\big(e^{-V(x+z)}-e^{-V(x)}\big)\,dz\bigg].
\endaligned
\end{equation*}
Moreover, for each $f\in C_c^{\infty}(\R^d)$, $L_{\psi,V,2}f \in
L^2(\mu_V)$.
\end{example}
\begin{proof}
It is easy to see that for any $f$, $g\in C_c^\infty(\R^d)$,
$$
\aligned D_{\psi,V}&(f,g)=\frac{1}{2}\iint{\big(f(x)\!-f(y)\big)\big(g(x)-g(y)\big)}j(x,y)\,\mu_V(dy)\,\mu_V(dx),
\endaligned$$
where
$$j(x,y)=\psi(|x-y|).$$
Following the computation in Example \ref{exm2.1}, we can check that
under (\ref{exm2.2.1}) and the condition $e^{-V}\in C_b^1(\R^d)$, the statements (\ref{thm2.1.2}) and (\ref{thm2.1.3})
hold. Then, applying Theorem \ref{thm2.1} (1) and (2), we know that
$(D_{\psi,V}, \mathscr{D}_2(D_{\psi,V}))$ is a symmetric Dirichlet
form on $L^2(\mu_{V})$ such that $C_b^\infty(\R^d)\subset
\mathscr{D}_2(D_{\psi,V})$, and for any $f$, $g\in
C_c^\infty(\R^d)$,
$$D_{\psi,V}(f,g)=-\int gL_{\psi,V,2}f\,d\mu_V,$$ where
\begin{equation*}
\aligned L_{\psi,V,2}f(x)=&{C_{V}}\bigg[\int \Big
(f(x+z)-f(x)-\nabla f(x)\cdot
z\I_{\{|z|\leqslant 1\}}\Big)\psi(|z|) e^{-V(x+z)}\,dz\\
&\qquad+\nabla f(x) \cdot \int_{\{|z|\leqslant 1\}} z
\psi(|z|)\big(e^{-V(x+z)}-e^{-V(x)}\big)\,dz\bigg].
\endaligned
\end{equation*}
On the other hand, by direct computation,
\begin{equation*}
\begin{split}
&I_{3,r}(x)=\I_{(B(0,2r))^c}(x)\int_{\{|z+x| \leqslant r\}}
\psi(|z|)e^{-V(x+z)}\,dz,
\end{split}
\end{equation*}
which is bounded with respect to $x$, and hence $I_{3,r} \in
L^2(\mu_V)$. It follows from Theorem \ref{thm2.1} (3) that the
operator $L_{\psi,V,2}$ maps $ C_c^\infty(\R^d)$ into $L^2(\mu_V)$.
\end{proof}

In order to drive weighted functional inequalities in the next two sections, we consider the following examples about
truncated Dirichlet forms. The proof is similar to that of Examples
\ref{exm2.1} and \ref{exm2.2}, and we omit the details here.
\begin{example}{\bf (Truncated Dirichlet Form for
$(D_{\rho,V}, \mathscr{D}( D_{\rho,V}))$)} \label{exm2.4} Let $\rho$ be
the same function as that in Example \ref{exm2.1} such that (\ref{exm2.1.1})
is satisfied. For all $r>0,$ define $\hat
\rho(r):=\rho(r)\I_{\{r>1\}}$. Consider the following truncated form
\begin{equation*}
\begin{split}
\hat D_{\rho,V}(f,g):=& \frac{1}{2}\iint
{\big(f(y)-f(x)\big)\big(g(y)-g(x)\big)}\hat \rho(|x-y|)\,dy\,\mu_V(dx)\\
=&\frac{1}{2}\iint_{\{|x-y|>1\}}
{\big(f(y)-f(x)\big)\big(g(y)-g(x)\big)}\rho(|x-y|)\,dy\,\mu_V(dx),\\
\mathscr{D}(\hat D_{\rho,V}):=&\bigg\{f\in L^2(\mu_V): \hat D_{\rho,
V}(f,f)<\infty\bigg\}.\end{split}
\end{equation*}
Then, $(\hat D_{\rho,V}, \mathscr{D}(\hat D_{\rho,V}))$ is a
symmetric Dirichlet form on $L^2(\mu_V)$ such that
$C_b^\infty(\R^d)\subset \mathscr{D}(\hat D_{\rho,V})$, and for any
$f$, $g\in C_c^\infty(\R^d)$,
\begin{equation}\label{exm2.4.1}
\hat D_{\rho,V}(f,g)=-\int g \hat L_{\rho,V}f\,d\mu_V,
\end{equation}
where
\begin{equation*}
\begin{split}
\hat L_{\rho,V}f(x)=\frac{1}{2}\int_{\{|z|>1\}} \big
(f(x+z)-f(x)\big)\rho(|z|) \big(e^{V(x)-V(x+z)}+1\big)\,dz.
\end{split}
\end{equation*}
Furthermore, if (\ref{exm2.1.4}) holds, then $\hat L_{\rho,V}f \in
L^2(\mu_V)$ for each $f\in C_c^{\infty}(\R^d)$.
\end{example}
\begin{remark}
For the truncated Dirichlet form $(\hat D_{\rho,V}, \mathscr{D}(\hat D_{\rho,V}))$, in order to let (\ref{exm2.4.1}) hold, we do not need the condition
$e^{-V}\in C_b^1(\R^d)$ as that in Example \ref{exm2.1}. The reason is as follows: for
truncated Dirichlet form $(\hat D_{\rho,V}, \mathscr{D}(\hat D_{\rho,V}))$, $I_2(x)=0$, and so it suffices to check that $I_1(x)$ belongs to $L^1(\mu_V)$ and is locally bounded. Hence, according to the proof of Example \ref{exm2.1},
the condition $e^{-V}\in C_b^1(\R^d)$ is not necessary, and the assumption that $e^{-V}$ is bounded and
$V$ is locally bounded is enough to ensure the function $I_1(x)$ belongs to $L^1(\mu_V)$ and is locally bounded.
\end{remark}

\begin{example}{\bf (Truncated Dirichlet Form for
$(D_{\psi,V}, \mathscr{D}( D_{\psi,V}))$)}\label{exm2.5} Let $\psi$ be
the same function as that in Example \ref{exm2.2} such that (\ref{exm2.2.1})
holds. For any $r>0$, define $\hat \psi(r):=\psi(r)\I_{\{r>1\}}$ and
a truncated form as follows
\begin{equation*}\label{e10b}
\begin{split}
\hat D_{\psi,V}(f,g):= &\frac{1}{2}\iint
{\big(f(y)-f(x)\big)\big(g(y)-g(x)\big)}\hat \psi(|x-y|)\,e^{-V(y)}\,dy\,e^{-V(x)}\,dx\\
=&\frac{1}{2}\iint_{\{|x-y|>1\}}
{\big(f(y)-f(x)\big)\big(g(y)-g(x)\big)}\psi(|x-y|)\,e^{-V(y)}\,dy\,e^{-V(x)}\,dx,\\
\mathscr{D}(\hat D_{\psi,V}):=&\bigg\{f\in L^2(\mu_{2V}): \hat
D_{\psi, V}(f,f)<\infty\bigg\}.\end{split}
\end{equation*}
Then, $(\hat D_{\psi,V}, \mathscr{D}(\hat D_{\psi,V}))$ is a
symmetric Dirichlet form on $L^2(\mu_{2V})$, which satisfies that
$C_c^\infty(\R^d)\subset \mathscr{D}(\hat D_{\psi,V})$, and for any
$f$, $g\in C_c^\infty(\R^d)$,
$$\hat D_{\psi,V}(f,g)=-\int g\hat L_{\psi,V}f\,d\mu_{2V}.$$ Here,
\begin{equation*}
 \hat L_{\psi,V}f(x)=\frac{1}{C_{2V}}\int_{\{|z|>1\}} \big (f(x+z)-f(x)
\big)\psi(|z|) e^{V(x)-V(x+z)}\,dz,
\end{equation*}
In particular, if (\ref{exm2.2.4}) holds, then $\hat L_{\psi,V}f \in
L^2(\mu_{2V})$ for each $f\in C_c^{\infty}(\R^d)$.
\end{example}

\begin{remark} The Dirichlet form in Example \ref{exm2.2} is associated with symmetric Markov processes under Girsanov transform of
pure jump type, see \cite[Theorem 3.4]{CZ} and \cite[Theorem
3.1]{S}. The Dirichlet from in Example \ref{exm2.3} was first
introduced in \cite{WJ2} to study the symmetric property of L\'{e}vy
type operators. In particular, if $\psi(r)={r^{-(d+\alpha)}}$ for constant $\alpha \in (0,2)$
in Example \ref{exm2.3}, then the Dirichlet from $(D_{\psi,V}, \mathscr{D}_2(D_{\psi,V}))$ has a non-local expression of
the Gagliardo semi-norms for fractional Sobolev spaces
$W^{\alpha/2,2}(\R^d)$ of order $\alpha/2$, see \cite{MRS}. Different from Examples \ref{exm2.2} and \ref{exm2.3}, the
Dirichlet form appearing in Example \ref{exm2.1} is of the ``not
symmetric'' expression.
\end{remark}

\section{Weighted Poincar\'{e} Inequalities for general Non-local Dirichlet Forms via Lyapunov
Conditions}\label{sec3} Let $j$ be a nonnegative and symmetric jump kernel on $\R^{2d}\setminus\{(x,y)\in\R^{2d}: x=y\}$, and $\mu_V$ be
a probability measure on $(\R^d, \mathscr{B}(\R^d))$ such that \eqref{thm2.1.2} and \eqref{thm2.1.3}
hold. Let $(D_{j,V},\mathscr{E}(D_{j,V}))$ be the regular Dirichlet
form given in the paragraph below the proof of Theorem \ref{thm2.1}. In order to
consider weighted Poincar\'{e} inequalities for
$(D_{j,V},\mathscr{E}(D_{j,V}))$, we start with the truncated
Dirichlet form $(\hat D_{j,V},\mathscr{E}(\hat D_{j,V}))$
corresponding to $(D_{j,V},\mathscr{E}(D_{j,V}))$; namely, for any
$f$, $g\in C_b^\infty(\R^d)$,
\begin{equation*}
\begin{split}
\hat
D_{j,V}(f,g):=&\frac{1}{2}\iint_{\{|x-y|>1\}}(f(x)-f(y))(g(x)-g(y))j(x,y)\,\mu_V(dy)\,\mu_V(dx),
\end{split}
\end{equation*} and  $\mathscr{E}(\hat D_{j,V})$ is the closure of
$C_b^{\infty}(\R^d)$ under the norm $$
\|f\|_{\hat{D}_{j,V},{1}}:=\big(\|f\|^2_{L^2(\mu_V)}+\hat
D_{j,V}(f,f)\big)^{1/2}.$$ According to the proof of Theorem
\ref{thm2.1} and the paragraph below it (also see Example \ref{exm2.4}),
for any $f$, $g\in C_c^\infty(\R^d)$,
\begin{equation*}
\hat D_{j,V}(f,g)=-\int g \hat L_{j,V}f\,d \mu_V,
\end{equation*}
where
\begin{equation}\label{sec3.1}
\hat L_{j,V}f(x)=\int_{\{|x-y|>1\}} \big (f(y)-f(x)
\big)j(x,y)\mu_V(dy).
\end{equation}

 Denote
by $\hat \Gamma_{j,V}(f,g)$ the carr\'e de champ of $(\hat
L_{j,V},C_c^\infty(\R^d))$, i.e.,\ for any $f$, $g\in
C_c^\infty(\R^d)$,
\begin{equation}\label{sec3.2}
\hat \Gamma_{j,V}(f,g)(x):=\hat L_{j,V}(fg)(x)-g(x)\hat L_{j,V}f(x)-f(x)\hat L_{j,V}g(x).
\end{equation}
It is easy to check that for each $f$, $g\in C_c^\infty(\R^d)$,
\begin{align*}\label{e12}
\hat \Gamma_{j,V}(f,g)(x)
=&\int_{\{|x-y|>1\}}\big(f(x)-f(y)\big)\big(g(x)-g(y)\big)j(x,y)\,\mu_V(dy).
\end{align*}
Furthermore, due to (\ref{thm2.1.2}) and (\ref{thm2.1.3}), we know
that for every $f$, $g\in C_c^\infty(\R^d)$, $\hat L_{j,V}f$ and
$\hat \Gamma_{j,V}(f,g)$ are well defined and locally bounded.

In order to apply Lyapunov functions (which usually are unbounded) to
$\hat L_{j,V}$, we need the following bigger domain associated with
$\hat L_{j,V}$ and $\hat \Gamma_{j,V}$:
\begin{equation}\label{sec3.3}
\aligned \hat{\mathscr{C}}_{j,V}:= \Big\{ \phi\in C^{\infty}(\R^d):\,\,&
x\mapsto
\int_{\{|x-y|>1\}} |\phi(y)|j(x,y)\mu_{V}\,(dy)\\ \ &\qquad\textrm{is locally
bounded}\Big\}.\endaligned
\end{equation}
First, by \eqref{sec3.1}, \eqref{sec3.2}, \eqref{sec3.3} and some direct
computation, we have the following
\begin{lemma}\label{lem3.1}
For each $\phi \in \hat{\mathscr{C}}_{j,V}$ and $f\in
C_c^{\infty}(\R^d)$, both $\hat L_{j,V}\phi$ and $\hat
\Gamma_{j,V}(f,\phi)$ are pointwise well defined by \eqref{sec3.1} and
\eqref{sec3.2}, respectively; moreover, both of them are locally
bounded, and
\begin{equation*}
\hat \Gamma_{j,V}(f,\phi)(x)
=\int_{\{|x-y|>1\}}\big(f(x)-f(y)\big)\big(\phi(x)-\phi(y)\big)j(x,y)\,\mu_V(dy)
\end{equation*}
\end{lemma}

Next, we shall use Lyapunov type conditions for the operator $\hat
L_{j,V}$. These conditions are known to yield functional
inequalities for Markov processes, e.g.\ Poincar\'{e} inequalities
for diffusion processes, see \cite{BCG}, and super-Poincar\'{e}
inequalities for symmetric Markov processes, see \cite{CGWW}. The
following lemma further shows that Lyapunov type conditions imply
functional inequalities for non-local symmetric Dirichlet forms. The
proof follows some method from \cite[Theorem 2.1]{WJ3}.

\begin{proposition}\label{pr3.1} Suppose that there exist two positive functions $\phi\in
\hat{\mathscr{C}}_{j,V}$, $h\in C^\infty(\R^d)$, and two constants
$b$, $r>0$ such that for all $x\in\R^d$,
\begin{equation}\label{pr3.1.1}  \hat L_{j,V}\phi(x)\leqslant -h(x)+b\I_{B(0,r)}(x).\end{equation} Then for any $f\in
C_c^\infty(\R^d)$,
\begin{equation}\label{pr3.1.2}
\int f^2h\phi^{-1}\,d\mu_V\leqslant \hat
D_{j,V}(f,f)+b\int_{B(0,r)}f^2\phi^{-1}\,d\mu_V.
\end{equation}
Moreover, if there is a constant $c>0$ such that $\phi\geqslant c$,
then the inequality \eqref{pr3.1.2} still holds for every $f\in
C_b^{\infty}(\R^d)$.
\end{proposition}
\begin{proof}
(a) We first consider the case for $f\in C_c^{\infty}(\R^d)$. By
Theorem \ref{thm2.1} and Lemma \ref{lem3.1}, for any $f\in
C_c^\infty(\R^d)$,
\begin{align*}\hat D_{j,V}(f,f)&=-\int f\hat L_{j,V}f\,d\mu_V=-\int
f\hat L_{j,V}\Big(\frac{f}{\phi}\phi\Big)\,d\mu_V\\
&=-\int\bigg(\frac{f^2}{\phi}\hat L_{j,V}\phi+f\phi
\hat L_{j,V}\Big(\frac{f}{\phi}\Big)+f\hat \Gamma_{j,V}\Big(\frac{f}{\phi},\phi\Big)\bigg)\,d\mu_V\\
&=-\int
f^2\frac{\hat L_{j,V}\phi}{\phi}\,d\mu_V+\bigg[\hat D_{j,V}\big(f\phi,\frac{f}{\phi}\big)-\int
f\hat \Gamma_{j,V} \big(\frac{f}{\phi},\phi\big)\,d\mu_V\bigg]\\
&=-\int
f^2\frac{\hat L_{j,V}\phi}{\phi}\,d\mu_V+\bigg(\frac{1}{2}\int\hat \Gamma_{j,V}\big(f\phi,\frac{f}{\phi}\big)\,d\mu_V-\int
f\hat \Gamma_{j,V} \big(\frac{f}{\phi},\phi\big)\,d\mu_V\bigg)\\
&=:\hat J_1(f)+\hat J_2(f),\end{align*} where in the equalities
above we have used the facts that for $f\in C_c^\infty(\R^d)$,
$f\phi^{-1}\in C_c^\infty(\R^d)$, and $\hat \Gamma_{j,V}
\big({f}{\phi}^{-1},\phi\big)$ is well defined and locally bounded,
thanks to $\phi\in \hat{\mathscr{C}}_{j,V}$.

Next, we will claim that $\hat J_2(f)\geqslant 0$. If this holds,
then we get
\begin{equation}\label{profpr3.1.0}\hat D_{j,V}(f,f)\geqslant -\int
f^2\frac{\hat L_{j,V}\phi}{\phi}\,d\mu_V,\end{equation} which, along with
\eqref{pr3.1.1}, immediately yields the inequality (\ref{pr3.1.2}) for
every $f\in C_c^{\infty}(\R^d)$.

To prove $\hat J_2(f)\geqslant 0$, it suffices to verify that for
all $f\in C_c^\infty(\R^d)$,
$$\hat J_3(f):=2\hat J_2(f\phi)=\int\hat \Gamma_{j,V}\big(f\phi^2,{f}\big)\,d\mu_V-2\int
f\phi\hat \Gamma_{j,V} \big({f},\phi\big)\,d\mu_V\ge0.$$ Note that
\begin{equation*}
\begin{split}
\hat J_3&(f)\\=&\iint_{\{|x-y|>1\}}\!\!\big(f(x)-f(y)\big)\big(f(x)\phi^2(x)-f(y)\phi^2(y)\big)j(x,y)\,\mu_V(dy)\,\mu_V(dx)\\
&-2\iint_{\{|x-y|>1\}}\!\!
f(x)\phi(x)\big({f}(x)-f(y)\big)\big(\phi(x)-\phi(y)\big)j(x,y)\,\mu_V(dy)\,\mu_V(dx)\\
=&\iint_{\{|x-y|>1\}}\!\!\big(f(x)-f(y)\big)\\
&\times\!\Big(f(x)\phi^2(x)-f(y)\phi^2(y)-2f(x)\phi(x)\big(\phi(x)-\phi(y)\big)\Big)\!\,j(x,y)\,\mu_V(dy)\,\mu_V(dx)\\
=&\iint_{\{|x-y|>1\}}\!\!\big(f(x)-f(y)\big)\\
&\times\Big(-f(x)\phi^2(x)-f(y)\phi^2(y)+2f(x)\phi(x)\phi(y)\Big)j(x,y)\,\mu_V(dy)\,\mu_V(dx),
\end{split}
\end{equation*}
In particular, due to (\ref{sec3.3}), every item above is finite, and
the case $\infty-\infty$ will not happen.

By the symmetric property that
$j(x,y)\,\mu_V(dy)\,\mu_V(dx)=j(y,x)\,\mu_V(dy)\,\mu_V(dx)$, we have
\begin{align*}
\hat J_3&(f)\\
=&\iint_{\{|x-y|>1\}}\big(f(y)-f(x)\big)\\
&\,\,\times\Big(-f(y)\phi^2(y)-f(x)\phi^2(x)+2f(y)\phi(y)\phi(x)\Big)\,j(x,y)\,\mu_V(dy)\,\mu_V(dx)\\
=&\iint_{\{|x-y|>1\}}\big(f(x)-f(y)\big)\\
&\,\,\times\Big(f(y)\phi^2(y)+f(x)\phi^2(x)-2f(y)\phi(y)\phi(x)\Big)\,j(x,y)\,\mu_V(dy)\,\mu_V(dx).
  \end{align*}
Adding the two equalities above, we arrive at
\begin{align*}\hat J_3(f)
=&2\iint_{\{|x-y|>1\}}\big(f(x)-f(y)\big)^2\phi(y)\phi(x)\,j(x,y)\,\mu_V(dy)\,\mu_V(dx)\geqslant
0.
\end{align*}
This proves (\ref{pr3.1.2}) for every $f\in C_c^{\infty}(\R^d)$.

(b) For every $f\in C_b^{\infty}(\R^d)$, there is a sequence of
functions $\{f_n\}_{n=1}^{\infty}\subseteq$ $C_c^{\infty}(\R^d)$
such that
\begin{equation}\label{profpr3.1.1}
\lim_{n\rightarrow \infty}f_n(x)=f(x),\ \
\sup_n\|f_n\|_{\infty}<\infty,\ \ \sup_n\|\nabla
f_n\|_{\infty}<\infty.
\end{equation}
Let $F_n(x,y):=\big(f_n(x)-f_n(y)\big)^2$. We have $$\hat
D_{j,V}(f_n,f_n)=\frac{1}{2}\iint_{\{|x-y|>1\}}F_n(x,y)j(x,y)\,\mu_V(dx)\,\mu_V(dy).$$
Note that (\ref{profpr3.1.1}) implies that $\sup_n|F_n(x,y)|<\infty$. Thus,
by (\ref{thm2.1.2}) and the dominated convergence theorem, we get
\begin{equation*}
\lim_{n\rightarrow \infty}\hat D_{j,V}(f_n,f_n)=
\hat D_{j,V}(f,f)
\end{equation*}
Since $\phi\geqslant c>0$, also due to the dominated convergence
theorem,
\begin{equation*}
\lim_{n\rightarrow \infty}\int f_n^2\phi^{-1}d\mu_V= \int
f^2\phi^{-1}d\mu_V.
\end{equation*}
On the other hand, thanks to the Fatou lemma, \begin{equation*} \int f^2
h\phi^{-1}d\mu_V\leqslant \liminf_{n \rightarrow \infty} \int f_n^2
h\phi^{-1}d\mu_V.
\end{equation*}
Since (\ref{pr3.1.2}) holds for each $f_n$, letting $n$ tend to
infinity and using the estimates above, we can show that
(\ref{pr3.1.2}) holds for $f\in C_b^{\infty}(\R^d)$.
\end{proof}

\begin{remark} (1)
If $\phi\geqslant c>0$, then, by taking $f=1$ in (\ref{pr3.1.2}), we have
$\int h \phi^{-1}d\mu_V<\infty$.

(2) In the proof of Proposition \ref{pr3.1} above, the main step is to show \eqref{profpr3.1.0},
which formally is the same as the conclusion of
\cite[Lemma2.12]{CGWW}. However, here we can not apply
\cite[Lemma2.12]{CGWW} directly, since the generator
$(\hat L_{j,V}, C_c^\infty(\R^d))$ associated with $\hat D_{j,V}$ is not necessarily self-joint in
$L^2(\mu_V)$, see the paragraph below the proof of Theorem \ref{thm2.1}. 
\end{remark}

Another ingredient is the following local Poincar\'{e} inequality:

\begin{lemma}\label{le3.1}For any $r>0$ and any $f\in C_b^\infty(\R^d)$,
\begin{equation}\label{le3.1.1}\int_{B(0,r)}f^2\,d\mu_V\leqslant \kappa_rD_{j,V}(f,f)
+\mu_V(B(0,r))^{-1}\Big(\int_{B(0,r)}f\,d\mu_V\Big)^2,\end{equation}
where
\begin{equation}\label{le3.1.2}
\kappa_r=\frac{1}{\mu_V(B(0,r))^2}\sup_{x\in B(0,r)}\int_{B(0,r)}j(x,y)^{-1}\,\mu_V(dy).
\end{equation}
\end{lemma}
\begin{proof} Note that, \eqref{le3.1.1} is equivalent to
$$\int_{B(0,r)}\bigg(f(x)-\frac{1}{\mu_V(B(0,r))}\int_{B(0,r)}f(x)\,\mu_V(dx)\bigg)^2\,\mu_V(dx)\le \kappa_rD_{j,V}(f,f).$$

For any $f\in C_b^\infty(\R^d)$, by the Cauchy-Schwarz inequality,
\begin{align*}\int_{B(0,r)}&\bigg(f(x)-\frac{1}{\mu_V(B(0,r))}\int_{B(0,r)}f(x)\,\mu_V(dx)\bigg)^2\mu_V(dx)\\
&=\int_{B(0,r)}\bigg(\frac{1}{\mu_V(B(0,r))}\int_{B(0,r)}(f(x)-f(y))\,\mu_V(dy)\bigg)^2\mu_V(dx)\\
&\le \frac{1}{\mu_V(B(0,r))^2}\int_{B(0,r)}\bigg(\int_{B(0,r)}(f(x)-f(y))^2j(x,y)\,\mu_V(dy)\bigg)\\
&\qquad\qquad\qquad\qquad\qquad\qquad\quad\times\bigg(\int_{B(0,r)}j(x,y)^{-1}\,\mu_V(dy)\bigg)\,\mu_V(dx)\\
&\le \frac{1}{\mu_V(B(0,r))^2}\bigg[\sup_{x\in B(0,r)}\bigg(\int_{B(0,r)}j(x,y)^{-1}\,\mu_V(dy)\bigg)\bigg]\\
&\qquad\qquad\qquad\quad\times\,\int_{B(0,r)}\bigg(\int_{B(0,r)}{(f(x)-f(y))^2}j(x,y)\,\mu_V(dy)\bigg)\mu_V(dx)\\
&\le\kappa_r D_{j,V}(f,f).\end{align*} The proof is complete.
 \end{proof}
\begin{remark}
 The item
$\int_{B(0,r)}j(x,y)^{-1}\,\mu_V(dy)$, and so the constant
$\kappa_r$, may not always be finite. For instance, if $j(x,y)=0$
for any $x$, $y\in\R^d$ with $|x-y|\leqslant 1$, then
$\int_{B(0,r)}j(x,y)^{-1}\,\mu_V(dy)=\infty$ for all $x\in\R^d$ with $|x|<r-1$.
\end{remark}

Having Proposition \ref{pr3.1} and Lemma \ref{le3.1} at hand, we are
in position to present the main result in this section.

\begin{theorem}\label{th3.1} Suppose that there exist $\phi\in
\hat{\mathscr{C}}_{j,V}$ with $\phi\geqslant 1$,  $h\in
C^\infty(\R^d)$ with $h>0$ and two constants $b$, $r_0>0$ such that
for all $x\in\R^d$,
\begin{equation}\label{th3.1.1} \hat L_{j,V}\phi(x)\leqslant -h(x)+b\I_{B(0,r_0)}(x).\end{equation}
If $\mu_V( \phi h^{-1})<\infty$ and the constant $\kappa_r$ defined
by \eqref{le3.1.2} satisfies $\kappa_r<\infty$ for each $r>1$, then
there exists a constant $C_1>0$ such that for any $f\in
C_b^\infty(\R^d)$ with $\int f d\mu_V=0$,
\begin{equation}\label{th3.1.2}
\int f^2h\phi^{-1}\,d\mu_V\leqslant C_1D_{j,V}(f,f).
\end{equation}
\end{theorem}

\begin{remark} If $\inf_{|x|\geqslant \eta}(h\phi^{-1})(x)>0$
for $\eta$ large enough, then $\sup_{|x|\geqslant
\eta}(h^{-1}\phi)(x)<\infty$, which yields that $\mu_V(
h^{-1}\phi)<\infty.$ In this case, the weighted Poincar\'{e}
inequality \eqref{th3.1.2} is stronger than the usual Poincar\'{e}
inequality.
\end{remark}
\begin{proof}[Proof of Theorem \ref{th3.1}]
According to \eqref{th3.1.1} and Proposition \ref{pr3.1}, for any
$f\in C_b^\infty(\R^d)$ with $\int f\,d\mu_V=0$,
$$\int f^2h\phi^{-1}\,d\mu_V\leqslant D_{j,V}(f,f)+b\int_{B(0,r_0)}f^2\phi^{-1}\,d\mu_V,$$
where we have used the fact that $\hat D_{j,V}(f,f) \leqslant
D_{j,V}(f,f)$. Since $\phi\geqslant 1$, by the local Poincar\'{e}
inequality \eqref{le3.1.1}, for any $r\ge r_0$,
\begin{align*}\int_{B(0,r_0)}f^2\phi^{-1}\,d\mu_V&\leqslant \int_{B(0,r_0)}f^2\,d\mu_V\\
&\leqslant \int_{B(0,r)}f^2\,d\mu_V\\
&\le \kappa_rD_{j,V}(f,f)+\frac{1}{\mu_V(B(0,r))}\Big(\int_{B(0,r)}f\,d\mu_V\Big)^2\\
&=\kappa_rD_{j,V}(f,f)+\frac{1}{\mu_V(B(0,r))}\Big(\int_{B(0,r)^c}f\,d\mu_V\Big)^2,\end{align*}
where in the equality above we have used the fact that
$$\int_{B(0,r)}f\,d\mu_V=-\int_{B(0,r)^c}f\,d\mu_V.$$

Using the Cauchy-Schwarz inequality, we find
\begin{align*}\Big(\int_{B(0,r)^c}f\,d\mu_V\Big)^2\leqslant &\Big(\int_{B(0,r)^c}f^2h\phi^{-1}\,d\mu_V\Big)\Big(\int_{B(0,r)^c}\phi h^{-1}\,d\mu_V\Big),\end{align*}
Therefore, for any $f\in C_b^\infty(\R^d)$ with $\int f\,d\mu_V=0$
and $r\geqslant r_0$,
$$\int f^2h\phi^{-1}\,d\mu_V\leqslant \big(1+b\kappa_r\big)D_{j,V}(f,f)+\frac{b\int_{B(0,r)^c}\phi h^{-1}\,d\mu_V}{\mu_V(B(0,r))}\int f^2h\phi^{-1}\,d\mu_V.$$
Since $\mu_V(\phi h^{-1})<\infty$, we can choose $r_1\ge r_0$ large
enough such that $$\frac{b\int_{B(0,r_1)^c}\phi
h^{-1}\,d\mu_V}{\mu_V(B(0,r_1))}\leqslant 1/2,$$ which gives us the
inequality (\ref{th3.1.2}) with $C_1=2(1+b\kappa_{r_1})$.
\end{proof}

\begin{remark}
In this section, we have showed the method on how to deduce
functional inequalities for $D_{j,V}$ from the Lyapunov conditions
for the generator associated with truncated Dirichlet form $\hat
D_{j,V}$. In fact, we can also use the Lyapunov conditions for
$D_{j,V}$ itself (see \cite{WW}), but more technical conditions in
the definition (\ref{sec3.3}) for the class $\hat{\mathscr{C}}_{j,V}$
are required. We will see in the next two sections that the Lyapunov
conditions for truncated Dirichlet form is efficient in a number
of applications.
\end{remark}

\section{Weighted Poincar\'{e} Inequalities for Non-local Dirichlet
Forms}\label{sec4}
\subsection{Weighted Poincar\'{e} Inequalities for Non-local Dirichlet
Forms $D_{\rho,V}$} Throughout this section, we assume that
$\rho:\R_+:=(0,\infty) \rightarrow \R_+$ satisfies (\ref{exm2.1.1}),
$\mu_V(dx)=C_Ve^{-V(x)}\,dx$ is a probability measure such that
$e^{-V}$ is bounded and $V$ is locally bounded, and $D_{\rho,V}$ is
the bilinear form defined by (\ref{exm2.1.2}). The following
statement presents the criterion about weighted Poincar\'{e}
inequalities for $D_{\rho,V}$.

\begin{theorem}\label{th4.1}
Suppose $\rho(r): \R_+ \rightarrow \R_+$ is a continuous function
satisfying \eqref{exm2.1.1},
\begin{equation}\label{th4.1.1}
\int_{\{r\le 1\}}{r^{d-1}}{\rho(r)^{-1}}\,dr<\infty,
\end{equation}
and
\begin{equation}\label{th4.1.2}
\int_{\{|x|>1\}} \frac{e^{-2V(x)}}{\gamma(|x|)}\,dx<\infty,
\end{equation} where for $r >0$, $$\gamma(r):=\inf_{0<s\leqslant r+1}\rho(s).$$

If there exists some constant $0<\alpha_0<1$ such that
\begin{equation}\label{th4.1.3}
\int_{\{r>1\}} r^{d+\alpha_0-1}\rho(r)\,dr<\infty
\end{equation}
and \begin{equation}\label{th4.1.4}\limsup_{|x|\to\infty}
\frac{\sup_{|z|\geqslant
|x|}e^{-V(z)}}{\gamma(|x|)|x|^{\alpha_0}}=0,
\end{equation}
then there exists a constant $C_1>0$ such that the following
weighted Poincar\'{e} inequality
\begin{equation*}
\int \big(f (x)-\mu_V(f)\big)^2
e^{V(x)}\gamma(|x|)\,\mu_V(dx) \leqslant C_1D_{\rho,V}(f,f)
\end{equation*}
holds for all $f\in C_b^\infty(\R^d)$.
\end{theorem}

To prove Theorem \ref{th4.1}, we begin with the
truncated form $\hat D_{\rho,V}$ associated with $D_{\rho,V}$. In
particular, according to Example \ref{exm2.4}, we have
\begin{lemma}\label{le4.2} For any $f$, $g\in C_c^\infty(\R^d)$,
\begin{equation*}
-\int g (\hat L_{\rho,V} f) \,d\mu_V=\hat D_{\rho,V}(f,g),
\end{equation*}
where
\begin{equation}\label{le4.2.1}
\hat L_{\rho,V}f(x)=\frac{1}{2}\int_{\{|z|>1\}} \big (f(x+z)-f(x)\big)\rho(|z|)
\big(e^{V(x)-V(x+z)}+1\big)\,dz.
\end{equation}
\end{lemma}

Next, we shall study Lyapunov type conditions for the operator $\hat
L_{\rho, V}$. The crucial step is the proper choice of the Lyapunov
function,
 see e.g.\ \cite{WJ1}.
The following lemma is motivated by the proof of \cite[Lemma
3.8]{WW}, which was used to prove super Poincar\'{e} inequalities
and Poincar\'{e} inequalities for fractional Dirichlet forms
$D_{\alpha,V}$ (see \cite[Theorem 3.6]{WW}).

\begin{lemma}\label{le4.3} Let $\alpha_0$ be the constant in Theorem $\ref{th4.1}$, and $\phi\in C^{\infty}(\R^d)$ be a
function such that $\phi \geqslant 1$ and $\phi(x)=1+|x|^{\alpha_0}$
for $|x|>1$. Suppose that \eqref{th4.1.3} and \eqref{th4.1.4} hold.
Then $\hat L_{\rho,V}\phi(x)$ is well defined by \eqref{le4.2.1} and
locally bounded; moreover, there exist $r_0$, $C_1$ and $C_2>0$ such
that
\begin{equation}\label{le4.3.1}
\hat L_{\rho,V}\phi(x)\leqslant -C_1e^{V(x)}\gamma(|x|)\phi(x)+C_2\I_{B(0,r_0)}(x),
\end{equation}
where $\gamma(r):=\inf_{0<s\leqslant r+1}\rho(s)$ as that in Theorem
$\ref{th4.1}$.
\end{lemma}

\begin{proof}
Throughout the proof, all the constants $c_i$ $(i\ge0)$ do not depend on $x$. Let $c_1:=\sup_{|x|\leqslant 1}\phi(x)$. We claim that
\begin{equation*}
x\mapsto\int_{\{|x-y|>1\}} |\phi(y)|\rho(|x-y|)\,dy
\end{equation*}
is locally bounded. In fact,
\begin{equation*}
\begin{split}
 \int_{\{|x-y|>1\}} |\phi(y)|\rho(|x-y|)\,dy &\leqslant
\int_{\{|x-y|>1\}} \big(c_1+1+|y|^{\alpha_0}\big)\rho(|x-y|)\,dy \\
&\leqslant \int_{\{|x-y|>1\}}
\big(c_1+1+|x|^{\alpha_0}+|x-y|^{\alpha_0}\big)
\rho(|x-y|)\,dy\\
&\leqslant c_2(1+|x|^{\alpha_0}),
\end{split}
\end{equation*}
where in the second inequality we have used the fact that
$$|x+y|^{\alpha_0}\leqslant \big(|x|+|y|\big)^{\alpha_0}\leqslant |x|^{\alpha_0}
+|y|^{\alpha_0},\quad \alpha_0\in(0,1), x,y\in\R^d,$$
and the last inequality follows from (\ref{th4.1.3}). Hence, the
claim is true, and this yields the first conclusion of
the Lemma.

In order to complete the proof, we only need to verify
(\ref{le4.3.1}) for large values of $|x|$. For $|x|$ large enough,
\begin{equation*}
\begin{split}
\int_{\{|z|> 1\}}\big(\phi(x+z)-\phi(x)\big)\,\rho(|z|)\,dz
&\leqslant \int_{\{|z|> 1\}}\big(c_1+|x+z|^{\alpha_0}-|x|^{\alpha_0}\big)\,\rho(|z|)\,dz\\
&\leqslant \int_{\{|z|>
1\}}(c_1+|z|^{\alpha_0})\,\rho(|z|)\,dz\\
&=c_3<\infty.
\end{split}
\end{equation*}

Moreover, for $|x|$ large enough,
\begin{align*}
\int_{\{|z|> 1\}}&\big(\phi(x+z)-\phi(x)\big)\,e^{V(x)-V(x+z)}\rho(|z|)\,dz\\
&\leqslant e^{V(x)}\int_{\{|x+z|\leqslant 1\}}\big(c_1-|x|^{\alpha_0}\big)\,e^{-V(x+z)}\rho(|z|)\,dz\\
&\quad+ e^{V(x)}\int_{\{|z|> 1, |x+z|\leqslant |x|, |x+z|>1\}}\big(|x+z|^{\alpha_0}-|x|^{\alpha_0}\big)\,e^{-V(x+z)}\rho(|z|)\,dz\\
&\quad+ e^{V(x)}\int_{\{|z|> 1, |x+z|> |x|\}}\big(|x+z|^{\alpha_0}-|x|^{\alpha_0}\big)
e^{-V(x+z)}\rho(|z|)\,dz\\
&\leqslant e^{V(x)}\int_{\{|x+z|\leqslant 1\}}\big(c_1-|x|^{\alpha_0}\big)\,e^{-V(x+z)}\rho(|z|)\,dz\\
&\quad+ e^{V(x)}\int_{\{|z|> 1, |x+z|> |x|\}}|z|^{\alpha_0}e^{-V(x+z)}\rho(|z|)\,dz\\
&\leqslant e^{V(x)}\bigg(\inf_{|z|\leqslant 1}e^{-V(z)}\bigg)\int_{\{|x+z|\leqslant 1\}}\big(c_1-|x|^{\alpha_0}
\big)\rho(|z|)\,dz\\
&\quad+ e^{V(x)}\bigg(\sup_{|z|\geqslant |x|}e^{-V(z)}\bigg)\int_{\{|z|> 1, |x+z|> |x|\}}|z|^{\alpha_0}\rho(|z|)\,dz\\
&\leqslant e^{V(x)}\bigg(\inf_{|z|\leqslant 1}e^{-V(z)}\bigg)\int_{\{|x+z|\leqslant 1\}}\big(c_1-|x|^{\alpha_0}\big)\,
\rho(|z|)\,dz\\
&\quad+ e^{V(x)}\bigg(\sup_{|z|\geqslant |x|}e^{-V(z)}\bigg)\int_{\{|z|> 1\}}|z|^{\alpha_0}\rho(|z|)\,dz\\
&\leqslant-c_4e^{V(x)}\gamma(|x|)\big(1+|x|^{\alpha_0}\big)+c_5 e^{V(x)}\bigg(\sup_{|z|\geqslant |x|}e^{-V(z)}\bigg)\\
&\leqslant -c_6e^{V(x)}\gamma(|x|)\big(1+|x|^{\alpha_0}\big).
\end{align*}
Here, in the first inequality we split the integral domain and use
the facts that for $|x|$ large enough, $$\{z:|x+z|\leqslant 1\}=
\{z:|z|>1,\ |x+z|\leqslant 1\},$$ and $$\{z:|z|>1, |x+z|>
|x|\}\subseteq \{z:|z|>1,|x+z|> 1\};$$ in the second inequality we
drop the second term in the first inequality since it is negative;
in the
fifth inequality we use (\ref{th4.1.3}) and the estimate that
$$\inf_{z:|x+z|\leqslant 1}\rho(|z|)\ge \inf_{z: |z|\le |x|+1}\rho(|z|)= \gamma(|x|);$$ and the
last inequality follows from \eqref{th4.1.4}.

Combining all the estimates above and using (\ref{le4.2.1}), we get
the inequality (\ref{le4.3.1}) for $|x|$ large enough. This
completes the proof.
\end{proof}

Now we give the proof of Theorem \ref{th4.1}.

\begin{proof}[Proof of Theorem \ref{th4.1}]From Example \ref{exm2.1},
$$j(x,y)=\frac{1}{2C_V}\rho(|x-y|)\big(e^{V(x)}+e^{V(y)}\big),$$
where $$C_V=\frac{1}{\int e^{-V(x)}\,dx}.$$  By (\ref{sec3.3}), we
define the class of functions $\hat{\mathscr{C}}_{\rho,V}$ for $\hat
D_{\rho,V}$ as follows
\begin{equation*}
\begin{split}
\hat{\mathscr{C}}_{\rho,V}:=\Big\{ f\in C^{\infty}(\R^d):\,\,&
x\mapsto\int_{\{|x-y|>1\}}|f(y)| \rho(|x-y|)\big(1+e^{V(x)-V(y)}\big)\,dy\\
&\quad \textrm{is well defined and locally bounded}\Big\}.
\end{split}
\end{equation*}
Let $\phi\geqslant 1$ be the test function in Lemma \ref{le4.3}.
According to the proof of Lemma \ref{le4.3}, it is easy to check
that under \eqref{th4.1.3}, $\phi\in\hat{\mathscr{C}}_{\rho,V}$.

 Due to (\ref{th4.1.1}), and the facts that $\rho(r)$ is positive and continuous
on $r\in [1,\infty)$ and $V$ is locally bounded, ${j(x,y)}^{-1}$ is
integrable in $B(0,2r)\times B(0,2r)$. Hence, for each $r>0$,
$\kappa_r$ defined by (\ref{le3.1.2}) is finite.

By Lemma \ref{le4.3}, there exist $r_0$, $C_1$ and $C_2>0$ such that
$$\hat L_{\rho,V}\phi(x)\leqslant -C_1e^{V(x)}\gamma(|x|)\phi(x)+C_2\I_{B(0,r_0)}(x).$$
That is, (\ref{th3.1.1}) holds with
$h(x)=C_1e^{V(x)}\gamma(|x|)\phi(x)$.

Note that (\ref{th4.1.2}) implies $\mu_V(\phi h^{-1})<\infty$.
Therefore, the desired assertion follows from Theorem \ref{th3.1}.
\end{proof}

\subsection{Weighted Poincar\'{e} Inequalities for $D_{\alpha,V,\delta}$ with $\delta>0$}
We are now in a position to present the proof of Theorem
\ref{thm1.1}.
\begin{proof}[Proof of Theorem \ref{thm1.1}]
Choose $\rho(r)={e^{-\delta r}}{r^{-(d+\alpha)}}$ with
$\alpha\in(0,2)$ and $\delta>0$ in Theorem \ref{th4.1}.  We know
that \eqref{th4.1.1} and \eqref{th4.1.3} hold. It follows from
\eqref{thm1.1.1} that \eqref{th4.1.4} is satisfied. On the other
hand, under \eqref{thm1.1.1}, we know that for $|x|$ large enough,
$$e^{-V(x)}\le \sup_{|z|\ge |x|}e^{-V(z)}\le
e^{-\delta|x|}|x|^{-d-\alpha+\alpha_0}.$$ Therefore, there exists a
constant $c_1>0$ such that
$$\int_{\{|x|\ge 1\}}|x|^{d+\alpha}e^{-2V(x)+\delta|x|}\,dx\le c_1
\int_{\{|x|\ge
1\}}|x|^{-d-\alpha+2\alpha_0}e^{-\delta|x|}\,dx<\infty,$$ and so
\eqref{th4.1.2} also holds. Combining all the conclusions above, we
get the required assertion.
\end{proof}

Next, we turn to the proof of Proposition \ref{prop1.3}.

\begin{proof}[Proof of Proposition \ref{prop1.3}]
For any $n\ge1$, define $g_n(x):=e^{\lambda (|x|\wedge n)}$, where
$\lambda>0$ is a constant to be determined later. Clearly, $g_n$ is
a Lipschitz continuous and bounded function. By the approximation
procedure in the proof of Proposition \ref{pr3.1}, we can apply the
function $g_n$ into the inequality (\ref{sec1.1.10}). Thus,
\begin{equation}\label{proprop1.3.1}
\aligned \int g_n^2(x)\,\mu_V(dx)&\leqslant C_2\iint
\frac{(g_n(x)-g_n(y))^2}{|y-x|^{d+\alpha}}e^{-\delta|y-x|}\,dy
\,\mu_V(dx)\\
&\qquad\qquad \qquad\qquad+\Big(\int
g_n(x)\,\mu_V(dx)\Big)^2.\endaligned
\end{equation}

First, it holds for each $N>1$ that
\begin{equation*}
\begin{split}
\int \frac{(g_n(x)-g_n(y))^2}{|y-x|^{d+\alpha}}e^{-\delta|y-x|}\,dy
\leqslant& \int_{\{|x-y|\leqslant N\}}
\frac{(e^{\lambda (|x|\wedge n)}-e^{\lambda (|y|\wedge n)})^2}{|y-x|^{d+\alpha}}e^{-\delta|y-x|}\,dy\\
& +
\int_{\{|x-y|> N\}}\frac{(e^{\lambda (|x|\wedge n)}-e^{\lambda (|y|\wedge n)})^2}{|y-x|^{d+\alpha}}e^{-\delta|y-x|}\,dy\\
=&:J_{1,N}(x)+J_{2,N}(x).
\end{split}
\end{equation*}
By the mean value theorem and the facts that for any $x$,
$y\in\R^d$, $n\ge1$, $$\big||x|\wedge n -|y|\wedge n\big|\leqslant
|x-y|,$$ and
$$\int \frac{1}{|z|^{d+\alpha-2}}e^{-\delta |z|}dz<\infty,$$ we know that for any $x\in\R^d$,
\begin{equation*}
\begin{split}
J_{1,N}(x) & \leqslant  \lambda e^{2\big[\lambda (|x|\wedge
n)+\lambda N\big]}
\int_{\{|x-y|\leqslant N\}}\frac{|y-x|^2}{|y-x|^{d+\alpha}}e^{-\delta|y-x|}\,dy\\
&\leqslant c_1 e^{2\big[\lambda (|x|\wedge n)+\lambda N\big]}\\
 &=c_1 e^{2\lambda N}\, e^{2\lambda (|x|\wedge n)}
\end{split}
\end{equation*}
holds for some constant $c_1>0$ independent of $n$, $\lambda$ and
$N$. On the other hand, since  $$\big||x|\wedge n -|y|\wedge
n\big|\leqslant |x-y|,$$ it holds that $$|y|\wedge n \leqslant
|x|\wedge n +|x-y|.$$ Hence, choosing $\lambda\in(0,{\delta}/{4})$,
we obtain that for any $x\in\R^d$
\begin{equation*}
\begin{split}
J_{2,N}(x) & \leqslant 2 \int_{\{|x-y|>N\}}\frac{\big(e^{2\lambda
(|x|\wedge n)}+e^{2\lambda (|y|\wedge n)}\big)}
{|x-y|^{d+\alpha}}e^{-\delta|y-x|}\,dy\\
&\leqslant 2 \int_{\{|x-y|>N\}}\frac{\big(e^{2\lambda (|x|\wedge
n)}+e^{2\lambda (|x|\wedge n)}e^{2\lambda|x-y|}\big)}
{|x-y|^{d+\alpha}}e^{-\delta|y-x|}\,dy\\
&\leqslant 2e^{2\lambda (|x|\wedge n)}
\int_{\{|z|>N\}}\frac{1+e^{2\lambda
|z|}}{|z|^{d+\alpha}}e^{-\delta|z|}\,dz\\
&\leqslant 2k(N)\,e^{2\lambda (|x|\wedge n)},
\end{split}
\end{equation*}
where $$k(N):=\int_{\{|z|>N\}}\frac{1+e^{{\delta}
|z|/{2}}}{|z|^{d+\alpha}}e^{-\delta|z|}\,dz\le k(1)<\infty.$$ Therefore,
\begin{equation*}
\begin{split}
\iint
\frac{(g_n(x)-g_n(y))^2}{|y-x|^{d+\alpha}}e^{-\delta|y-x|}\,dy\,\mu_V(dx)
\leqslant&  \big(c_1 e^{2\lambda N}+2k(N)\big)\,\int e^{2\lambda
(|x|\wedge n)}\,\mu_V(dx).
\end{split}
\end{equation*}

Second, for any $n\ge 1$ and $\lambda>0$, set $$l_n(\lambda):=\int
g_n^2(x)\,\mu_V(dx)=\int e^{2\lambda (|x|\wedge n)}\,\mu_V(dx).$$
Then, combining all the estimates above with (\ref{proprop1.3.1}),
for each $\lambda\in(0,{\delta}/{4})$,
\begin{equation*}
\begin{split}
& l_n(\lambda)\leqslant C_2\big(c_1\lambda e^{2\lambda
N}+2k(N)\big)l_n(\lambda)+l_n^2({\lambda}/{2}).
\end{split}
\end{equation*}

Furthermore, using the Cauchy-Schwarz inequality, for any $R>1$, we
have
\begin{equation}\label{proprop1.3.2}
\begin{split}
l_n^2({\lambda}/{2})&\leqslant \Big(e^{\lambda R }+
\int_{\{|x|>R\}}e^{\lambda (|x|\wedge n)}\,\mu_V(dx)\Big)^2\leqslant
2e^{2\lambda R}+2p(R)\,l_n(\lambda),
\end{split}
\end{equation}
where $p(R):=\mu_V(|x|>R)$. Therefore, for each $N$, $R>0$ and
$\lambda\in(0,{\delta}/{4})$,
\begin{equation}\label{proprop1.3.3}
l_n(\lambda)\leqslant \big(C_2(c_1\lambda e^{2\lambda N}+2k(N))+2p(R)\big)l_n(\lambda)+2e^{2\lambda R},
\end{equation}

Now, we fix $R_0$ and $N_0>0$ large enough such that
$p(R_0)<{1}/{8}$ and $C_2k(N_0)<{1}/8$, and then take
$\lambda_0\in(0,\delta/4)$ small enough such that $C_2c_1\lambda_0
e^{2\lambda_0 N_0}<{1}/{4}$. Then, by (\ref{proprop1.3.3}), we
arrive at
\begin{equation*}
l_n(\lambda_0)\leqslant 8e^{2\lambda_0 R_0}.
\end{equation*}
Letting $n \rightarrow \infty$, we obtain the desired assertion.
\end{proof}
\subsection{Weighted Poincar\'{e} Inequalities for $D_{\alpha,V}$}
In the subsection, we will present the proofs of Theorem
\ref{thm1.4}, Propositions \ref{prop1.6.0} and
 \ref{prop1.6} .

\begin{proof} [Proof of Theorem \ref{thm1.4}]
(a) Choose $\rho(r)={r^{-(d+\alpha)}}$ with $\alpha\in(0,2)$ in
Theorem \ref{th4.1}.  It is easy to see that \eqref{th4.1.1},
\eqref{th4.1.3} and \eqref{th4.1.4} hold. On the other hand, under
\eqref{thm1.4.1}, we know that for $|x|$ large enough,
$$e^{-V(x)}\le \sup_{|z|\ge |x|}e^{-V(z)}\le
|x|^{-d-\alpha+\alpha_0},$$ and so there is a constant $c_2>0$ such
that for all $x\in\R^d$ with $|x|\ge 1$, $$e^{-V(x)}\le c_2
|x|^{-d-\alpha+\alpha_0}.$$
 Therefore, since
$\alpha_0\in(0,\alpha/2)$,
$$\int_{\{|x|\ge 1\}}|x|^{d+\alpha}e^{-2V(x)}\,dx\le c_2
\int_{\{|x|\ge 1\}}|x|^{-d-\alpha+2\alpha_0}\,dx<\infty.$$  That is,
\eqref{th4.1.2} also holds. The first required assertion follows
from all the conclusions above.

(b) Now, we will verify the second assertion. Suppose that the
inequality \eqref{thm1.4.2} holds with the weighted function
$\omega^*(x)$. Then, for all $f\in C_b^\infty(\R^d)$,
\begin{equation}\label{prothm1.4.1}\aligned
\int\big(f(x)-\mu_V(f)\big)^2&\,\omega^*(x)\,\mu_V(dx)\leqslant
C_3D_{\alpha,V}(f,f).\endaligned
\end{equation}
For any $n\geqslant 1$, choose a smooth function $f_n$:
$\R^d\to[0,1]$ such that
$$
    f_n(x)=
    \begin{cases}
0, & |x|\le n;\\
1,       & |x|>2n,
    \end{cases}
$$ and $|\nabla f_n(x)|\leqslant 2n^{-1}$
for all $x\in\R^d.$  Therefore, for all $x\in\R^d$,
\begin{align*}\Gamma(f_n)(x)&:=\int \frac{(f_n(y)-f_n(x))^2}{|y-x|^{d+\alpha}}\,dy\\
&\le \frac{4}{n^2}\int_{\{|y-x|\leqslant
n\}}\frac{1}{|y-x|^{d+\alpha-2}}\,dy+\int_{\{|y-x|\geqslant
n\}}\frac{1}{|y-x|^{d+\alpha}}\,dy\\
&\leqslant c_1n^{-\alpha},\end{align*} and
$$D_{\alpha,V}(f_n,f_n)=\mu_V\big(\Gamma(f_n)\big)\le
c_1n^{-\alpha},$$ where $c_1$ is a constant independent of $n$.

Then, there exists a constant $c_2>0$ independent of  $n$, such that for any $n$ large enough,
\begin{equation}\label{prothm1.4.2}
\begin{split}
&\int\big(f_n(x)-\mu_V(f_n)\big)^2\,\omega^*(x)\,\mu_V(dx)\\
&\geqslant\bigg(\inf_{|x|\geqslant 2n} \frac{\omega^*(x)}{\omega(x)}\bigg)\int_{\{|x|\geqslant 2n\}}\big(f_n(x)-\mu_V(f_n)\big)^2\,\omega(x)\,\mu_V(dx)\\
&\geqslant \bigg(\inf_{|x|\geqslant 2n} \frac{\omega^*(x)}{\omega(x)}\bigg)\bigg(1-\int_{\{|x|\geqslant n\}}\,\mu_V(dx)\bigg)^2 \int_{\{|x|\geqslant 2n\}}\omega(x)\,\mu_V(dx)\\
&\geqslant c_2\bigg(\inf_{|x|\geqslant 2n}
\frac{\omega^*(x)}{\omega(x)}\bigg) n^{-\alpha},
\end{split}
\end{equation}
where in the second inequality we have used the fact that
 $$0\leqslant \Big(1-\int_{\{|x|\ge n\}}\mu_V(dx)\Big) \leqslant\big(1-\mu_V(f_n)\big).$$
Thus, applying $f_n$ into (\ref{prothm1.4.1}), we get that there
exists some constant $c_3>0$ independent of $n$, such that for $n$
large enough,
\begin{equation*}
\bigg(\inf_{|x|\geqslant 2n} \frac{\omega^*(x)}{\omega(x)}\bigg)
n^{-\alpha}\leqslant c_3 n^{-\alpha}.
\end{equation*}
Since
$$\liminf_{n \rightarrow \infty}\frac{\omega^*(x)}{\omega(x)}=\infty,$$
there is a contradiction, and hence the conclusion is proved.
\end{proof}
\begin{proof}[Proof of Proposition \ref{prop1.6.0}]
We first claim that if the inequality (\ref{e1}) holds, then $\mu_V(\omega)<\infty$. In fact,
choose a function $g\in C_c^{\infty}(\R^d)$ such that $g(x)=0$ for every $|x|\ge 1$ and
$\mu_V(g)=1$. Then, applying this test function $g$ into (\ref{e1}), we have
\begin{equation*}
\begin{split}
\int_{\{|x|\ge 1\}}\omega(x)\,\mu_V(dx)&\le \int \big(g(x)-\mu_V(g)\big)^2\omega(x)\,\mu_V(dx)\\
&\leqslant C_0 D_{\alpha,V}(g,g)\\
&<\infty,
\end{split} 
\end{equation*} thanks to the fact that $g\in C_c^\infty(\R^d)\subset \mathscr{D}(D_{\alpha,V})$. 
Since the weighted function $\omega$ is continuous, it is bounded on $\{x \in \R^d: |x|\le 1\}$, and hence
$\int_{\{|x|\le 1\}}\omega(x)\,\mu_V(dx)<\infty$. Combining both estimates above, we prove the desired claim.

For any $t>1$ and $f \in C_b^{\infty}(\R^d)$, by (\ref{e1}), we have
\begin{equation*}
\begin{split}
\int_{\{|x|\ge t\}}f^2(x)\,\mu_V(dx)&\le \frac{1}{\inf\limits_{|x|\ge t}\omega(x)}\int f^2(x)\omega(x)\, \mu_V(dx)\\
&\le \frac{2}{\inf\limits_{|x|\ge t}\omega(x)}\int \big(f(x)-\mu_V(f)\big)^2\omega(x)\, \mu_V(dx)\\
&\quad+ \frac{2}{\inf\limits_{|x|\ge t}\,\omega(x)}\int \mu_V^2(f)\omega(x)\, \mu_V(dx)\\
&\le \frac{2}{\inf\limits_{|x|\ge t}\omega(x)}\bigg(C_0D_{\alpha,V}(f,f)+
{\mu_V(\omega)}\,\mu_V^2(|f|)\bigg),
\end{split}
\end{equation*} where in the second inequality we have used the fact that for any $a$, $b\in\R$, $$a^2\le 2(a-b)^2+2b^2.$$
Since $\lim_{|x|\to\infty} \omega(x)=\infty$, we can obtain that there exists a constant $c_1>0$ such that for any $t>1$,
$$\int_{\{|x|\ge t\}}f^2(x)\,\mu_V(dx)\le \frac{2C_0}{\inf\limits_{|x|\ge t}\omega(x)}D_{\alpha,V}(f,f)+c_1\mu_V(|f|)^2.$$

 On the other hand, according to \cite[Lemma 3.1]{WW}, for any $t
 > 1$, the following local super
Poincar\'{e}  inequality
\begin{equation*}
\begin{split}
& \int_{\{|x|\le t\}}f^2(x)\,\mu_V(dx)\le sD_{\alpha,V}(f,f)+\frac{c_2H(t)^{2+{d}/{\alpha}}}
{h(t)^{1+{d}/{\alpha}}}\big(1+s^{-{d}/{\alpha}}\big)\mu_V(|f|)^2,\quad s>0
\end{split}
\end{equation*}
holds with some constant $c_2>0$ independent of $t$.

Combining both estimates above, we get that there is a constant $c_3>0$ such that for each $t>1$ and $f\in C_b^\infty(\R^d)$,
\begin{equation}\label{e3}
\begin{split}
\mu_V(f^2)\le& \Big(\frac{2C_0}{\inf_{|x|\ge t}\omega(x)}+s\Big)D_{\alpha,V}(f,f)\\
&\qquad\qquad+\frac{c_3H(t)^{2+{d}/{\alpha}}}
{h(t)^{1+{d}/{\alpha}}}\big(1+s^{-{d}/{\alpha}}\big)\mu_V(|f|)^2,\quad  s>0.
\end{split}
\end{equation} This, along with the assumption that $\lim_{|x|\rightarrow \infty}\omega(x)=\infty$, yields the first required assertion.

For any $0<r<4C_0\big({\inf\limits_{x \in \R^d}\omega(x)}\big)^{-1}$,
we can choose $t>1$ large enough such that $2C_0\big({\inf_{|x|\ge t}\omega(x)}\big)^{-1}\le {r}/{2}$, e.g.\
$t=\kappa(4C_0/{r})$. Then, the second desired assertion follows by taking $s={r}/{2}$ and $t=\kappa(4C_0/{r})$ in the definition of rate function $\beta(r).$
\end{proof}

\begin{proof}[Proof of Proposition \ref{prop1.6}] Since $\omega$ is positive and continuous on $\R^d$, for any $r>0$, $\inf_{|x|\le r}\omega(x)>0$. For any $x\in\R^d$, set
$\omega^*(x)=\inf_{|z|\le |x|}\omega(z)$. Then, $\omega^*(x)\le \omega(x)$ for all $x\in\R^d$, and so it suffices to prove (\ref{prop1.6.2}) holds for
such weighted function $\omega^*$.

It is easy to check that, under (\ref{prop1.6.1}) the function $V$
satisfies all the conditions in Theorem \ref{thm1.4}. Therefore, the
inequality (\ref{thm1.4.2}) holds. If
$\lim_{|x|\to\infty}\omega^*(x)>0,$ then we can choose a constant
$C_0>0$ such that $\omega^*(x)\ge C_0$ for all $x\in\R^d$. Hence,
(\ref{thm1.4.2}) implies that (\ref{prop1.6.2}) holds for such
weighted function $\omega^*$.

In the following, we assume that $\omega^*$ is a positive function
on $\R^d$ such that
\begin{equation}\label{proprop1.6.1}\lim_{|x|\to\infty}\omega^*(x)=0.\end{equation}
For any $r>1$, which will be determined later, define a function
$\omega^*_r(x)$ as follows
\begin{equation*}
\omega^*_r(x):=\begin{cases}
& \quad 1,\ \ \ \ \ \text{if}\ \ |x|\leqslant r,\\
& \omega^*(x),\ \ \ \text{if}\ \ |x|>r.
\end{cases}
\end{equation*}
Clearly, there is a constant $c_1:=c_1(r)>0$ such that
$\omega^*_r(x)\leqslant c_1 \omega^*(x)$ for each $x\in \R^d$. Therefore, it is sufficient to prove (\ref{prop1.6.2}) holds for
weighted function $\omega^*_r$ with some $r>0$.

Let $V^*(x):=V(x)-\log \omega^*_r(x)$ (we omit the index $r$ in
$V^*$ for simplicity). By (\ref{prop1.6.1}) and
(\ref{proprop1.6.1}), we can check that the function $V^*$ satisfies
all the conditions in Theorem \ref{thm1.4}. Thus, there exists a
constant $c_2:=c_2(r)>0$ such that the following weighted
Poincar\'{e} inequality
\begin{equation}\label{proprop1.6.2}
\aligned \int \big(f (x)-\mu_{V^*}(f)\big)^2&
\frac{1}{(1+|x|)^{d+\alpha}}\,dx\leqslant c_2\iint \frac{(f(y)-f(x))^2}{|y-x|^{d+\alpha}}\,dy
\,\mu_{V^*}(dx)\endaligned
\end{equation}
holds for all $f\in C_b^\infty(\R^d)$, where
$$\mu_{V^*}(dx)=\frac{1}{\int e^{-V^*}(x)\,dx}e^{-V^*(x)}\,dx=:C_{V^*}e^{-V^*(x)}\,dx.$$

Since $r>1$ and $w^*_r(x)=1$ for all $|x|\le 1$,
\begin{equation}\label{proprop1.6.3}
C_{V^*}=\frac{1}{\int e^{-V^*(x)}dx}\leqslant
\frac{1}{\int_{\{|x|\leqslant 1\}} e^{-V(x)}\,dx}=:C_1<\infty,
\end{equation}
where $C_1>0$ is a constant independent of $r$. According to
(\ref{proprop1.6.1}), we can fix $r$ large enough such that $\sup_{x
\in \R^d}|\omega^*_r(x)|\leqslant 1$. Hence, there exists a constant
$C_2>0$ independent of $r$, such that for every $f\in
C_b^{\infty}(\R^d)$ with $\int f d\mu_V=0$,
\begin{equation}\label{proprop1.6.4}
\aligned
\Big(\int& f(x)  \,\mu_{V^*}(dx)\Big)^2\\
&=\Big(C_{V^{*}}\int_{\{|x|\leqslant r\}} f(x) e^{-V(x)}\,dx+
C_{V^{*}}\int_{\{|x|>r\}}f(x)\omega^*(x)e^{-V(x)}\,dx\Big)^2\\
&=\Big(C_{V^{*}}\int_{\{|x|> r\}} f(x) e^{-V(x)}\,dx+
C_{V^{*}}\int_{\{|x|>r\}}f(x)\omega^*(x)e^{-V(x)}\,dx\Big)^2\\
&\le 2C_1^2\bigg[\Big(\int_{\{|x|> r\}} f(x) e^{-V(x)}\,dx\Big)^2+
\Big(\int_{\{|x|>r\}}f(x)\omega^*(x)e^{-V(x)}\,dx\Big)^2\bigg]\\
&\leqslant C_2\bigg[ \mu_V(|x|>r)\int f^2(x)\,\mu_V(dx)\\
&\qquad\qquad\qquad\qquad
+ \Big(\int_{\{|x|>r\}}{\omega^*}^2(x)\,\mu_V(dx)\Big)\Big(\int f^2(x)\,\mu_V(dx)\Big)\bigg]\\
&\leqslant 2C_2\mu_V(|x|>r)\int f^2(x)\,\mu_V(dx),\endaligned
\end{equation}
where the second equality follows from
$$\int_{\{|x|\leqslant r\}}f(x)e^{-V(x)}\,dx=\int_{\{|x|> r\}}f(x)e^{-V(x)}\,dx,$$ thanks to $\int f \,d\mu_V=0$; in the first inequality we have used (\ref{proprop1.6.3}); in the second inequality we applied the Cauchy-Schwarz inequality; and the last inequality follows from the fact that
$\sup_{x \in \R^d}|\omega^*_r(x)|\leqslant 1$.

From now on, all the constants $C_i$ are independent of $r$. For
each $f\in C_b^{\infty}(\R^d)$ with
$\int f d\mu_V=0$,
\begin{equation*}
\begin{split}
 \int &f^2(x) \frac{e^{V(x)}}{(1+|x|)^{d+\alpha}}\,\mu_V(dx)\\
&\leqslant C_3 \int\Big(f(x)-\int f \,d\mu_{V^*}+
\int f \,d\mu_{V^*} \Big)^2 \frac{1}{(1+|x|)^{d+\alpha}}\,dx\\
&\leqslant C_4\int\Big(f(x)-\int f \,d\mu_{V^*}\Big)^2 \frac{1}{(1+|x|)^{d+\alpha}}\,dx
+C_4\Big(\int f(x) \,\mu_{V^*}(dx)\Big)^2\\
&\leqslant C_4\int\Big(f(x)-\int f \,d\mu_{V^*}\Big)^2 \frac{1}{(1+|x|)^{d+\alpha}}\,dx
+C_5\mu_V(|x|>r)\int f^2(x)\,\mu_V(dx)\\
&\leqslant C_4\int\Big(f(x)-\int f \,d\mu_{V^*}\Big)^2 \frac{1}{(1+|x|)^{d+\alpha}}\,dx\\
&\qquad+C_6\mu_V(|x|>r)\int f^2(x)\frac{e^{V(x)}}{(1+|x|)^{d+\alpha}}\,\mu_V(dx),
\end{split}
\end{equation*}
where in the third inequality we have used (\ref{proprop1.6.4}), and
the last inequality follows from (\ref{prop1.6.1}).

Choose $r$ large enough such that $C_6\mu_V(|x|>r)<{1}/{2}$. Then, for each $f\in C_b^{\infty}(\R^d)$ with
$\int f d\mu_V=0$,
\begin{equation*}
\int f^2(x) \frac{e^{V(x)}}{(1+|x|)^{d+\alpha}}\,\mu_V(dx)\leqslant 2C_4
\int\Big(f^2(x)-\int f d\mu_{V^*}\Big)^2 \frac{1}{(1+|x|)^{d+\alpha}}\,dx.
\end{equation*}
This, along with (\ref{proprop1.6.2}), yields that for each $f\in
C_b^{\infty}(\R^d)$ with $\int f d\mu_V=0$,
$$\aligned \int f^2(x) \frac{e^{V(x)}}{(1+|x|)^{d+\alpha}}\,\mu_V(dx)\leqslant c_3\int\omega_r^*(x)\int
\frac{(f(y)-f(x))^2}{|y-x|^{d+\alpha}}\,dy
\,\mu_{V}(dx)\endaligned
$$
This proves the desired assertion.\end{proof}

\section{Weighted Poincar\'{e} Inequalities for Non-local Dirichlet Forms
Associated with Symmetric Markov
Processes under Girsanov Transform of Pure Jump Type}
\subsection{Weighted Poincar\'{e} Inequalities for Dirichlet Forms $D_{\psi,V}$}
In this section, we aim to state that the technique to yield Theorem
\ref{th4.1} also gives us the criterion for weighted Poincar\'{e}
inequalities of the Dirichlet form given in Example \ref{exm2.2}. In
the following, let $V$ be a locally bounded function on $\R^d$ such
that $\int \,e^{-V(x)}\,dx$ $<\infty$ and $e^{-V}$ is bounded. Let
\begin{equation}\label{sec5.1.1}\mu_{2V}(dx):=\frac{1}{\int\,e^{-2V(x)}\,dx}\,e^{-2V(x)}\,dx\end{equation}
be a probability measure. Given a measurable function $\psi(r):\R_+
\rightarrow \R_+$ satisfying \eqref{exm2.2.1}, we define for each
$f$, $g\in C_c^\infty(\R^d)$,
$$D_{\psi,V}(f,g):=\frac{1}{2}\iint
{\big(f(y)-f(x)\big)\big(g(y)-g(x)\big)}\psi(|x-y|)\,e^{-V(y)}\,dy\,e^{-V(x)}\,dx.$$
\begin{theorem}\label{thm5.1} Assume that $\psi:\R_+
\rightarrow \R_+$ is a continuous function such that
\begin{equation*}
\int_{\{|r|\le 1\}}r^{d-1}\psi(r)^{-1}\,dr<\infty,
\end{equation*}
and \begin{equation*} \int_{\{|x|\ge1\}}
\frac{e^{-3V(x)}}{\gamma(|x|)}\,dx<\infty,
\end{equation*} where for $r>0$, $$\gamma(r):=\inf_{0<s\leqslant r+1}\psi(s).$$

If there is some constant $0<\alpha_0<1$ such that
\begin{equation*}
\int_{\{r\ge1\}} r^{d+\alpha_0-1}\psi(r)\,dr<\infty,
\end{equation*}
 and
\begin{equation*}\limsup_{|x|\to\infty}
\frac{\sup_{|z|\geqslant
|x|}e^{-V(z)}}{\gamma(|x|)|x|^{\alpha_0}}=0,
\end{equation*}
then there exists a constant $C_1>0$ such that the following
weighted Poincar\'{e} inequality
\begin{equation*}
\int \big(f (x)-\mu_{2V}(f)\big)^2 e^{V(x)}\gamma(|x|)\,\mu_{2V}(dx)
\leqslant C_1D_{\psi,V}(f,f)
\end{equation*}
holds for all $f\in C_b^\infty(\R^d)$.
\end{theorem}
The proof of Theorem \ref{thm5.1} is similar to that of Theorem
\ref{th4.1}. It is based on the expression for the generator of the
associated truncated Dirichlet form in Example \ref{exm2.5}, which
enables us to take the same Lyapunov function as that in the Lemma
\ref{le4.3}. We omit the details here.

\subsection{The Case that: $\psi(r)={e^{-\delta r}}{r^{-(d+\alpha)}}$
with $\delta> 0$ and $0<\alpha<2$} Taking $\psi(r)={e^{-\delta
r}}{r^{-(d+\alpha)}}$ with $\delta >0$ and $0<\alpha<2$ in Theorem
\ref{thm5.1}, we have the following statement.

\begin{corollary}\label{cor5.2}
Suppose that for some constants $\delta > 0$, $\alpha\in(0,2)$
 and $\alpha_0\in(0,1)$,
\begin{equation*}\limsup_{|x|\to\infty}\bigg[
\bigg(\sup_{|z|\geqslant |x|}e^{-V(z)}\bigg)e^{\delta
|x|}|x|^{d+\alpha-\alpha_0}\bigg]=0.
\end{equation*}
Then, there exists a constant $C_0>0$ such that the following
weighted Poincar\'{e} inequality
\begin{equation}\label{cor5.2.1}
\aligned \int \big(f (x)-&\mu_{2V}(f)\big)^2
\frac{e^{V(x)-\delta|x|}}{(1+|x|)^{d+\alpha}}\,\mu_{2V}(dx)\\
&\leqslant C_0\iint \frac{(f(y)-f(x))^2}{|y-x|^{d+\alpha}}e^{-\delta
|y-x|}e^{-V(y)}dy \,e^{-V(x)}dx\endaligned
\end{equation}
holds for all $f\in C_b^\infty(\R^d)$.
\end{corollary}

According to Corollary \ref{cor5.2}, we know that if
$$\liminf_{|x|\to\infty}\frac{e^{V(x)}}{|x|^{d+\alpha}e^{\delta|x|}}>0,$$then
\eqref{cor5.2.1} holds, which implies the standard Poincar\'{e}
inequality, see \eqref{prop5.3.1} below. On the other hand, we can
prove the following statement, which indicates the concentration of
measure for such Poincar\'{e} inequality. Roughly speaking, in this
setting, under some regular assumptions on $V$, the exponential
integrability of the distance function is also necessary for this
inequality.
\begin{proposition}\label{prop5.3}
Let $\delta>0$ and $\alpha\in(0,2)$, and let $\mu_{2V}$ be a
probability measure defined by \eqref{sec5.1.1} such that $V \in
C^1(\R^d)$ and $\sup_{x \in \R^d}|\nabla V(x)|$ $<\delta$. If there
is a constant $C_1>0$ such that the following inequality
\begin{equation}\label{prop5.3.1}
\begin{split}
\int \big(f (x)-\mu_{2V}(f)\big)^2&\,
\mu_{2V}(dx)\\
&\leqslant C_1\iint \frac{(f(y)-f(x))^2}{|y-x|^{d+\alpha}}e^{-\delta
|y-x|}e^{-V(y)}\,dy \,e^{-V(x)}\,dx
\end{split}
\end{equation}
holds for all $f \in C_b^{\infty}(\R^d)$, then there exists a constant $\lambda_0>0$ such that
$$\int e^{\lambda_0 |x|}\mu_{2V}(dx)<\infty.$$
\end{proposition}
\begin{proof}
For any $n\ge1$, define $g_n(x):=e^{\lambda (|x|\wedge n)}$, where
$\lambda>0$ is a constant to be determined later. As the same reason
as that in the proof of Proposition \ref{prop1.3}, we can apply
$g_n$ into (\ref{prop5.3.1}), and get that
\begin{equation}\label{proprop5.3.1}
\begin{split}
\int g_n^2(x)\,\mu_{2V}(dx)&\leqslant C_1 \iint
\frac{(g_n(y)-g_n(x))^2}{|y-x|^{d+\alpha}}e^{-\delta
|y-x|}e^{-V(y)}\,dy \,e^{-V(x)}\,dx \\
&\qquad\qquad+\Big(\int g_n(x)\, \mu_{2V}(dx)\Big)^2.
\end{split}
\end{equation}

In the following, for $\lambda>0$, set $$l_n(\lambda):=\int
e^{2\lambda (|x|\wedge n)}\,\mu_{2V}(dx)=\int
g_n^2(x)\,\mu_{2V}(dx).$$ For each $N>1$ and all $x\in\R^d$,
\begin{equation*}
\begin{split}
\int
\frac{(g_n(x)-g_n(y))^2}{|y-x|^{d+\alpha}}&e^{-\delta|y-x|}e^{-V(y)}\,dy \\
\leqslant &\int_{\{|x-y|\leqslant N\}}
\frac{(e^{\lambda (|x|\wedge n)}-e^{\lambda (|y|\wedge n)})^2}{|y-x|^{d+\alpha}}e^{-\delta|y-x|}e^{-V(y)}\,dy \\
&+ \int_{\{|x-y|> N\}}\frac{(e^{\lambda (|x|\wedge n)}-e^{\lambda
(|y|\wedge n)})^2}{|y-x|^{d+\alpha}}e^{-\delta|y-x|}
e^{-V(y)}\,dy\\
=&:J_{1,N}(x)+J_{2,N}(x).
\end{split}
\end{equation*}
From now on, the constant $C$ will be changed in different lines,
but does not depend on $n$, $N$, $\lambda$ or $R$.
Let $b:=\sup_{x \in \R^d}|\nabla V(x)|$. As in the proof
of Proposition \ref{prop1.3}, by the mean value theorem,
\begin{equation*}
\begin{split}
 J_{1,N}(x) &\leqslant  \lambda e^{2\lambda (|x|\wedge n)+2\lambda
N}
\int_{\{|x-y|\leqslant N\}}\frac{|y-x|^2}{|y-x|^{d+\alpha}}e^{-\delta|y-x|}e^{-V(y)}dy\\
&\leqslant \lambda e^{2\lambda (|x|\wedge n)+2\lambda
N}e^{-V(x)+b N}
\int_{\{|x-y|\leqslant N\}}\frac{|y-x|^2}{|y-x|^{d+\alpha}}e^{-\delta|y-x|}dy\\
&\leqslant C\lambda  e^{2\lambda (|x|\wedge n)+2\lambda
N}e^{-V(x)+bN}\\
&= C\lambda e^{2\lambda N+bN}e^{2\lambda (|x|\wedge n)-V(x)},
\end{split}
\end{equation*}
where in the second inequality we have used the fact that $$-V(y)=-V(x)+(V(x)-V(y))\le -V(x)+|V(x)-V(y)|\le -V(x)+ b|x-y|.$$ Hence,
$$\int J_{1,N}(x) e^{-V(x)}\,dx \le C\lambda e^{2\lambda N+bN}\int e^{2\lambda (|x|\wedge n)}\,\mu_{2V}(dx)=
C\lambda e^{2\lambda N+bN}l_n(\lambda).$$ On the other hand, by
symmetric property,
\begin{equation}\label{proprop5.3.2}
\begin{split}
\int J_{2,N}(x)&  e^{-V(x)}\,dx \\
& \leqslant 2\iint_{\{|x-y|>N\}}\frac{e^{2\lambda (|x|\wedge
n)}+e^{2\lambda (|y|\wedge n)}}
{|x-y|^{d+\alpha}}e^{-\delta|x-y|}e^{-V(y)}\,dy e^{-V(x)}\,dx\\
&\leqslant 4 \iint_{\{|x-y|>N\}} \frac{e^{2\lambda (|x|\wedge
n)}} {|x-y|^{d+\alpha}}e^{-\delta|x-y|}e^{-V(y)}\,dye^{-V(x)}\,dx.
\end{split}
\end{equation}
Since $-V(y)\le -V(x)+|V(x)-V(y)|\leqslant -V(x)+b|x-y|$, it holds that
\begin{equation*}
\begin{split}
\int_{\{|x-y|>N\}}\frac{e^{-\delta|x-y|}}
{|x-y|^{d+\alpha}}e^{-V(y)}\,dy
&\leqslant e^{-V(x)}\int_{\{|x-y|>N\}}\frac{e^{-\delta|x-y|}}
{|x-y|^{d+\alpha}}e^{b |x-y|}\,dy\\
&=k(N)e^{-V(x)},
\end{split}
\end{equation*}
where $$k(N):=\int_{\{|x-y|>N\}}\frac{e^{-(\delta-b)|z|}}
{|z|^{d+\alpha}}dz\le k(1)<\infty,$$ also thanks to the fact that
$b<\delta$. Combining this with (\ref{proprop5.3.2}), we find
\begin{equation*}
\int J_{2,N}(x) e^{-V(x)}\,dx \leqslant Ck(N)\int e^{2\lambda
(|x|\wedge n)} e^{-2V(x)}\,dx\leqslant Ck(N)l_n(\lambda).
\end{equation*}

According to all the estimates above and (\ref{proprop5.3.1}), we
have
\begin{equation*}
l_n(\lambda)\leqslant C(\lambda e^{2\lambda
N+bN}+k(N))l_n(\lambda)+l_n^2({\lambda}/{2}).
\end{equation*}
As the same way in the proof of \eqref{proprop1.3.2}, for any $R>1$,
it holds that
\begin{equation*}
l_n^2({\lambda}/{2}) \leqslant 2e^{2\lambda R}+2p(R)l_n(\lambda),
\end{equation*}
where $p(R):=\mu_{2V}(|x|>R)$. Then,
\begin{equation*}
l_n(\lambda)\leqslant C(\lambda e^{2\lambda
N+bN}+k(N)+p(R))l_n(\lambda)+2e^{2\lambda R}.
\end{equation*}

Now, we first fix  $R_0$ and $N_0>0$ large enough such that
$Cp(R_0)<{1}/{4}$ and $Ck(N_0)<{1}/{4}$, then choose a constant
$\lambda_0>0$ such that
$C\lambda_0 e^{2\lambda_0 N_0+bN_0}<{1}/{4}$. We can finally get that
\begin{equation*}
l_n(\lambda_0)\leqslant 8e^{2\lambda_0R_0}.
\end{equation*}
Letting $n$ tends to $\infty$, we can prove the conclusion.
\end{proof}

\begin{remark}
 Let $\delta>0$ be the constant in the Poincar\'{e} inequality (\ref{prop5.3.1}). For any $\varepsilon>d$, let $\mu_{2\varepsilon}(dx)={C_{\varepsilon}}{(1+|x|)^{-2\varepsilon}}\,dx$ be a probability measure. We will claim that the Poincar\'{e} inequality (\ref{prop5.3.1}) does not hold for $\mu_{2\varepsilon}$ with any
$\varepsilon>d$. Indeed, for any $l\ge1$ and $\varepsilon>d$, define a probability measure
$$\mu_{l,2\varepsilon}(dx)={C_{l,\varepsilon}}{(l+|x|^2)^{-\varepsilon}}\,dx=:C_{l,\varepsilon} e^{-2V_{l,\varepsilon}(x)}\,dx.$$  We can choose $l_0$ large enough such that $\sup_{x\in \R^d}|\nabla V_{l_0,\varepsilon}(x)|<\delta$. Thus, according to Proposition
\ref{prop5.3}, the Poincar\'{e} inequality (\ref{prop5.3.1}) does
not hold for $\mu_{l_0,2\varepsilon}$. Then, the desired claim
follows from that fact that there is a constant $C:=C(l_0)>1$ such
that
$$\frac{1}{C}\mu_{l_0,2\varepsilon}\leqslant \mu_{2\varepsilon}
\leqslant C\mu_{l_0,2\varepsilon}.$$ Similarly, we also can show
that the Poincar\'{e} inequality (\ref{prop5.3.1}) does not hold for
$\mu_{2\beta}(dx)=C_{\beta}e^{-(1+|x|^\beta)}\,dx$ with any
$0<\beta<1$.
\end{remark}

\subsection{The Case that: $\psi(r)={r^{-(d+\alpha)}}$
with $0<\alpha<2$} Letting $\psi(r)={r^{-(d+\alpha)}}$ with
$0<\alpha<2$ in Theorem \ref{thm5.1}, we have the following
statement.

\begin{corollary}\label{cor5.5}
If for some constants $\alpha\in(0,2)$
 and $\alpha_0\in(0,\alpha\wedge1)$,
\begin{equation*}\limsup_{|x|\to\infty}\bigg[
\bigg(\sup_{|z|\geqslant
|x|}e^{-V(z)}\bigg)|x|^{d+\alpha-\alpha_0}\bigg]=0,
\end{equation*}
then there exists a constant $C_0>0$ such that the following
weighted Poincar\'{e} inequality
\begin{equation}\label{cor5.5.1}
\aligned \int \big(f (x)-\mu_{2V}(f)\big)^2&
\frac{e^{V(x)}}{(1+|x|)^{d+\alpha}}\,\mu_{2V}(dx)\\
&\leqslant C_0\iint \frac{(f(y)-f(x))^2}{|y-x|^{d+\alpha}}e^{-V(y)}dy \,e^{-V(x)}dx\endaligned
\end{equation}
holds for all $f\in C_b^\infty(\R^d)$.
\end{corollary}
To show that the inequality \eqref{cor5.5.1} is optimal, we consider
the following result. First, for each $\varepsilon>0$, let
$$V_{\varepsilon}(x):=\frac{1}{2}\log\Big[(1+|x|^2)^{d+\varepsilon}\Big],$$
and
$$\mu_{2V_\varepsilon}(dx):=C_{\varepsilon} e^{-2V_{\varepsilon}(x)}\,dx=\frac{C_{\varepsilon}}{(1+|x|^2)^{d+\varepsilon}}\,dx,$$
where $C_{\varepsilon}$ is the normalizing constant.
\begin{proposition}
The following Poincar\'{e} inequality
\begin{equation}\label{prop5.6.1}
\begin{split}
\int \big(f (x)&-\mu_{2V_\varepsilon}(f)\big)^2\,
\mu_{2V_\varepsilon}(dx)\\
&\leqslant C_1\iint
\frac{(f(y)-f(x))^2}{|y-x|^{d+\alpha}}e^{-V_{\varepsilon}(y)}\,dy \,
e^{-V_{\varepsilon}(x)}\,dx \quad\textrm{ for all }f \in
C_b^{\infty}(\R^d)\end{split}
\end{equation} holds some  constant $C_1>0$ if and only if $$\varepsilon\geqslant \alpha.$$
Moreover,  for the constant $\beta\in\R$ and the probability measure $\mu_{2V_\varepsilon}$ with $\varepsilon\geqslant\alpha$, the following weighted
Poincar\'{e} inequality
\begin{equation}\label{prop5.6.2}
\begin{split}
\int &\big(f (x)-\,\mu_{2V_\varepsilon}(f)\big)^2\,
\big(1+|x|^{\beta}\big)\,\mu_{2V_\varepsilon}(dx)\\
&\leqslant C_2\iint
\frac{(f(y)-f(x))^2}{|y-x|^{d+\alpha}}e^{-V_{\varepsilon}(y)}\,dy \,
e^{-V_{\varepsilon}(x)}\,dx\quad\textrm{ for all } f \in
C_b^{\infty}(\R^d)
\end{split}
\end{equation} holds for some constant $C_2>0$ if and only if $$\beta\leqslant \varepsilon-\alpha.$$

\end{proposition}
\begin{proof} (a)
According to Corollary \ref{cor5.5}, if $\varepsilon\geqslant
\alpha$, then the inequality (\ref{prop5.6.1}) holds for
$\mu_{2V_\varepsilon}$. Next, it suffices to verify that
(\ref{prop5.6.1}) does not hold for $\mu_{2V_\varepsilon}$ with
$\varepsilon< \alpha$. For any $n>1$, choose a smooth function $f_n:
\R^d\rightarrow [0,1]$ such that
\begin{equation*}
f_n(x)=\begin{cases}
& 0,\ \ \ \ \ \ \ \ \ \ \ \text{if} \ |x|\leqslant 3n;\\
& 1,\ \ \ \ \ \ \ \ \ \ \ \text{if}\ |x|>4n,
\end{cases}
\end{equation*}
and $\sup_{x \in \R^d}|\nabla f(x)|\leqslant {2}/{n}$. Suppose
(\ref{prop5.6.1}) holds for some probability measure
$\mu_{2V_\varepsilon}$ with $\varepsilon< \alpha$. Then,
\begin{equation}\label{proprop5.6.1}
\begin{split}
\int \Big(f_n (x)-\int f_n(x)\,&\mu_{2V_\varepsilon}(dx)\Big)^2\,
\mu_{2V_\varepsilon}(dx)\\
&\leqslant C_1\iint \frac{(f_n(y)-f_n(x))^2}{|y-x|^{d+\alpha}}e^{-V_{\varepsilon}(y)}\,dy \,
e^{-V_{\varepsilon}(x)}\,dx,
\end{split}
\end{equation}

In the following, set $$\Gamma(f_n)(x):=\int \frac{(f_n(y)-f_n(x))^2}{|y-x|^{d+\alpha}}e^{-V_{\varepsilon}(y)}\,dy.$$
The constant $C>0$ will be changed in different line, but does not depend on $n$.
When $|x|\leqslant 2n$, $|f_n(x)-f_n(y)|\neq 0$ only if $|x-y|>n$ and $|y|>3n$. Then,
\begin{equation*}
\begin{split}
&\Gamma(f_n)(x)\leqslant C\int_{\{|x-y|>n,|y|>3n\}}\frac{1}{|y-x|^{d+\alpha}}
\frac{1}{(1+|y|)^{d+\varepsilon}}dy\leqslant \frac{C}{n^{d+\alpha+\varepsilon}}.
\end{split}
\end{equation*}
For $2n<|x|\leqslant 5n$, it holds that $\{y:|x-y|\leqslant n\}\subseteq \{y:|y|\geqslant n\}$, and so,
\begin{equation*}
\begin{split}
\Gamma(f_n)(x) &\leqslant C\bigg[\int_{\{|x-y|\leqslant
n,|y|\geqslant n\}}\frac{(f_n(y)-f_n(x))^2}{|y-x|^{d+\alpha}}
e^{-V_{\varepsilon}(y)} \,dy\\
&\qquad\quad+\int_{\{|x-y|>n\}}\frac{(f_n(y)-f_n(x))^2}{|y-x|^{d+\alpha}}
e^{-V_{\varepsilon}(y)} \,dy\bigg]\\
&\leqslant C\bigg[\frac{1}{n^{2}}\int_{\{|x-y|\leqslant n,|y|\geqslant n\}}\frac{|x-y|^2}{|y-x|^{d+\alpha}}e^{-V_{\varepsilon}(y)}\, dy\\
&\qquad\quad
+\int_{\{|x-y|>n\}}\frac{1}{|y-x|^{d+\alpha}}
e^{-V_{\varepsilon}(y)} \,dy\bigg]\\
&\leqslant
C\bigg[\frac{1}{n^{d+2+\varepsilon}}\int_{\{|x-y|\leqslant
n\}}\frac{|x-y|^2}{|y-x|^{d+\alpha}} \,dy
+\frac{1}{n^{d+\alpha}}\int
e^{-V_{\varepsilon}(y)}\, dy\bigg] \\
&\leqslant \frac{C}{n^{d+\alpha}}
\end{split}
\end{equation*}
If $|x|>5n$, then $|f_n(x)-f_n(y)|\neq 0$ only for $|x-y|>n$, and hence
\begin{equation*}
\begin{split}
&\Gamma(f_n)(x)\leqslant C\int_{\{|x-y|>n\}}\frac{1}{|y-x|^{d+\alpha}}
e^{-V_{\varepsilon}(y)}\,dy
\leqslant \frac{C}{n^{d+\alpha}}\int e^{-V_{\varepsilon}(y)}\,dy
\leqslant \frac{C}{n^{d+\alpha}}.
\end{split}
\end{equation*}
Combining all the estimates above, we get
\begin{equation*}
\begin{split}
\int \Gamma(f_n)(x)&e^{-V_{\varepsilon}(x)}\,dx\\
=&\int_{\{|x|\leqslant 2n\}}\Gamma(f_n)(x)e^{-V_{\varepsilon}(x)}\,dx\\
&+\int_{\{2n<|x|\leqslant 5n\}}\Gamma(f_n)(x)e^{-V_{\varepsilon}(x)}\,dx+
\int_{\{|x|>5n\}}\Gamma(f_n)(x)e^{-V_{\varepsilon}(x)}\,dx\\
\leqslant& \frac{C}{n^{d+\alpha+\varepsilon}}\int e^{-V_{\varepsilon}(x)}\,dx
+\frac{C}{n^{d+\alpha}}\int_{\{|x|>2n\}}e^{-V_{\varepsilon}(x)}\,dx\\
\leqslant& \frac{C}{n^{d+\alpha+\varepsilon}}.
\end{split}
\end{equation*}
On the other hand, for $n$ large enough, following the proof of
(\ref{prothm1.4.2}), we have
\begin{equation*}
\begin{split}
\int\bigg(f_n(x)-&\int f_n(x)\,\mu_{2V_\varepsilon}(dx)\bigg)^2\,\mu_{2V_\varepsilon}(dx)\\
&\geqslant C\bigg(1-\int_{\{|x|\geqslant 3n\}}\,\mu_{2V_\varepsilon}(dx)\bigg)^2 \int_{\{|x|\geqslant 4n\}}e^{-2V_{\varepsilon}(x)}\,dx\\
&\geqslant \frac{C}{n^{d+2\varepsilon}}.
\end{split}
\end{equation*}

Therefore, according to \eqref{proprop5.6.1}, it holds for $n$ large
enough that
\begin{equation*}
\frac{1}{n^{d+2\varepsilon}}\leqslant \frac{C}{n^{d+\alpha+\varepsilon}}.
\end{equation*}
Since $\varepsilon<\alpha$, there is a contradiction, and hence
(\ref{prop5.6.1}) does not hold for $\mu_{2V,\varepsilon}$ with
$\varepsilon<\alpha$. We have proved the first conclusion.

\medskip

(b) For $\varepsilon \geqslant \alpha$, by Corollary \ref{cor5.5},
we can check that (\ref{prop5.6.2}) holds with $\beta\leqslant
\varepsilon-\alpha$. On the other hand, one can follow the argument
above to verify that (\ref{prop5.6.2}) does not hold with $\beta>
\varepsilon-\alpha$. This finished the proof.
\end{proof}

To compare the Dirichlet forms given in Examples \ref{exm2.2} and
\ref{exm2.3}, we will show that the corresponding Poincar\'{e}
inequality for the Dirichlet form given in Example \ref{exm2.3} does not
hold for a large class of probability measures.

\begin{proposition}\label{prop5.7}
Let $\psi:\R_{+}\rightarrow \R_{+}$ be a positive function
satisfying \eqref{exm2.2.1}. Let $\mu_V$ be a probability measure
 defined by \eqref{sec1.1.6} such that $V$ is a locally bounded function, $ e^{-V}\in
L^1(dx)$, $e^{-V}$ is bounded and $\lim_{|x|\rightarrow
\infty}e^{-V(x)}=0$. If there is a constant $C_1>0$ such that the
following Poincar\'{e} inequality
\begin{equation}\label{prop5.7.1}
\begin{split}
\int \big(f (x)-\mu_{V}(f)\big)^2&\,
\mu_{V}(dx)\\
&\leqslant C_1\iint (f(y)-f(x))^2 \psi(|x-y|)\,\mu_V(dy)
\,\mu_V(dx)
\end{split}
\end{equation}
holds for all $f \in C_b^{\infty}(\R^d)$, then  for any $\lambda>0$,
$$\int e^{\lambda |x|}\mu_V(dx)<\infty.$$

Furthermore, for any $r>0$, set $q(r):=\sup_{\{|x|>r\}}e^{-V(x)}$.
Then, there exist $C_2$, $C_3>0$ such that
\begin{equation}\label{prop5.7.2}
\int \bigg(\frac{1}{q(C_2|x|)}\bigg)^{C_3|x|}\mu_V(dx)<\infty.
\end{equation}
\end{proposition}

\begin{remark} (1) According to (\ref{prop5.7.2}), the
probability measures
$$\mu_\varepsilon(dx)={C_\varepsilon}{(1+|x|)^{-d-\varepsilon}}\,dx$$
for any $\varepsilon>0$ or $$\mu_\beta(dx)=C_\beta
e^{-(1+|x|^\beta)}\,dx$$ for any $\beta>0$ do not satisfy the
Poincar\'{e} inequality (\ref{prop5.7.1}).

(2) As seen above, the conclusions about the concentration of measure for functional
inequalities of Dirichlet forms given in Examples \ref{exm2.1} and
\ref{exm2.2} highly depend on the function $\rho$
or $\psi$ in these Dirichlet forms. However, here for the Dirichlet
form given in Example \ref{exm2.3}, the property of the
concentration of measure for Poincar\'{e} inequality is independent
of $\psi$.
\end{remark}

\begin{proof}[Proof of Proposition \ref{prop5.7}] Similar to the proof above, the constant $C$ will be changed in
different lines, but does not depend on $n$, $\lambda$, $M$, $N$ or
$R$. For each $\lambda>0$ and $n\ge1$, set $$g_n(x):=e^{\lambda
(|x|\wedge n)}.$$ Suppose that the inequality (\ref{prop5.7.1})
holds. Then, we may and do apply $g_n$ into \eqref{prop5.7.1} to get
that
\begin{equation}\label{proprop5.7.1}
\begin{split}
\int g_n^2(x)\,\mu_{V}(dx) \leqslant &C_1 \iint
(g_n(y)-g_n(x))^2\psi(|x-y|)\,\mu_V(dy)\,\mu_V(dx)\\
& \qquad\qquad+\Big(\int g_n(x)\,\mu_{V}(dx)\Big)^2.
\end{split}
\end{equation}

For any $n\ge1$ and $\lambda>0$, set $$l_n(\lambda):=\int
e^{2\lambda (|x|\wedge n)}\mu_{V}(dx)=\int g_n^2(x)\,\mu_{V}(dx).$$
For each $N>1$,
 we have
\begin{equation*}
\begin{split}
\int (g_n(x)-g_n(y))^2&\psi(|x-y|)e^{-V(y)}\,dy\\ \leqslant&
\int_{\{|x-y|\leqslant N\}}
(g_n(x)-g_n(y))^2\psi(|x-y|)e^{-V(y)}\,dy \\
&+
\int_{\{|x-y|> N\}}(g_n(x)-g_n(y))^2\psi(|x-y|)e^{-V(y)}\,dy\\
=&:J_{1,N}(x)+J_{2,N}(x).
\end{split}
\end{equation*}
For $r>0$, set $$a(r):=\int_{\{|z|\leqslant r\}}|z|^2\psi(|z|)\,dz$$
which is well defined due to (\ref{exm2.2.1}).

First, by using the mean
value theorem, we have
\begin{equation*}
\begin{split}
J_{1,N}(x)&\leqslant \lambda e^{2\lambda N}e^{2\lambda (|x|\wedge
n)}\int_{\{|x-y|\leqslant N\}} |x-y|^2\psi(|x-y|)e^{-V(y)}dy\\
&\le \lambda a(N)e^{2\lambda N}\,e^{2\lambda (|x|\wedge n)},
\end{split}
\end{equation*} where we have used the fact that
$e^{-V}$ is bounded. For any $r>0$, define
$$q(r):=\sup_{{x\in\R^d:|x|>r}}e^{-V(x)}.$$ Then, for $M>N$ large
enough,
\begin{equation*}
\begin{split}
 &\int J_{1,N}(x)\,\mu_V(dx)\\
& = \int_{\{|x|\leqslant M\}}J_{1,N}(x)\,\mu_V(dx)
+\int_{\{|x|>M\}}J_{1,N}(x)\,\mu_V(dx)\\
&\leqslant C\lambda e^{2\lambda N}a(N)\int_{\{|x|\leqslant M\}}e^{2\lambda (|x|\wedge n)}\,\mu_V(dx) \\
&\,\,+\!\lambda e^{2\lambda N}\!
\bigg[\int_{\{|x|>M\}} \!e^{2\lambda (|x|\wedge n)}\Big(\int_{\{|x-y|\leqslant N,|y|>M-N\}}\!\! |x-y|^2\psi(|x-y|)e^{-V(y)}\,dy\Big)e^{-V(x)}\,dx\bigg]\\
&\leqslant  C\lambda e^{2\lambda( M + N)}a(N) +C\lambda e^{2\lambda
N}a(N)q(M-N)l_n(\lambda),
\end{split}
\end{equation*} where in the first inequality we have used the fact
that if $|x|>M$ and $M>N$ large enough, then $$\{y:|x-y|\leqslant
N\}\subseteq \{y:|x-y|\leqslant N,\ |y|>M-N\}.$$

 On the other hand, by the
symmetric property and also the fact that $e^{-V}$ is bounded,
\begin{equation*}
\begin{split}
 \int J_{2,N}(x)\, \mu_V(dx) &\leqslant
2\iint_{\{|x-y|>N\}}\big(e^{2\lambda (|x|\wedge n)}+e^{2\lambda
(|y|\wedge n)}\big)
\psi(|x-y|)\,\mu_V(dy) \,\mu_V(dx)\\
&\leqslant 4 \iint_{\{|x-y|>N\}} e^{2\lambda (|x|\wedge n)}
\psi(|x-y|)\,\mu_V(dy)\, \mu_V(dx)\\
&\leqslant Ck(N)l_n(\lambda),
\end{split}
\end{equation*}
where $$k(N):=\int_{\{|x|>N\}}\psi(|z|)\,dz.$$

Combining all the estimates above with (\ref{proprop5.7.1}),
\begin{equation*}
\begin{split}
l_n(\lambda)\leqslant C\big(\lambda e^{2\lambda
N}a(N)q(M-N)+k(N)\big)l_n(\lambda)+l_n^2({\lambda}/{2}) +C\lambda
e^{2\lambda (M + N)}a(N).
\end{split}
\end{equation*}
By  (\ref{proprop1.3.2}), we get for each $R>1$,
\begin{equation*}
l_n(\lambda)\leqslant C\big(\lambda e^{2\lambda
N}a(N)q(M-N)+k(N)+p(R)\big)l_n(\lambda)+ C\big(\lambda e^{2\lambda
(M +N)}a(N)+e^{2\lambda R}\big),
\end{equation*}
where $p(R):=\mu_V(|x|>R)$.

Next, we first choose the constants $R_0$ and $N_0>0$ large enough
such that
$$C(k(N_0)+p(R_0))<{1}/{2}.$$ Since $\lim_{r \rightarrow
\infty}q(r)=0$, for all $\lambda>0$, we can take
$M_0:=M_0(\lambda)>0$ large enough such that
$$C\lambda e^{2\lambda N_0}a(N_0)q(M_0-N_0)\leqslant {1}/{4},$$ e.g.
\begin{equation*}
M_0=N_0+q^{-1}\Big(\frac{e^{-2\lambda N_0}}{4C(\lambda \vee1)
a(N_0)}\Big),
\end{equation*}
where \begin{equation}\label{proprop5.7.2}q^{-1}(r):=\inf\{s: \
q(s)\leqslant r\},\end{equation}
and we use the convention that $q^{-1}(r):=0$ if $r>q(0)$. Then, we have
\begin{equation*}
l_n(\lambda)\leqslant C\big(\lambda e^{2\lambda (M_0 +
N_0)}a(N_0)+e^{2\lambda R_0}\big).
\end{equation*}
Letting $n \rightarrow \infty$, we get $\int e^{\lambda
|x|}\mu_V(dx)<\infty$ for every $\lambda>0$. This proves the first
required assertion.

Actually, according to the arguments above, there are constants
$c_1$, $c_2>0$ such that
\begin{equation}\label{proprop5.7.3}
\int e^{\lambda |x|}\,\mu_V(dx)<c_1 e^{c_1\big(\lambda+\lambda
q^{-1}(e^{-c_2\lambda}) \big)}.
\end{equation}
Now, we will follow the proof of \cite[Theorem 3.3.21]{WBook}.
Define
$$w(\lambda):=e^{-(c_1+1)\lambda-c_1\lambda q^{-1}(e^{-c_2\lambda})
}$$ and
$$h(r):=\int_1^{+\infty}e^{r\lambda}w(\lambda)\,d\lambda.$$ Then, by Fubini theorem and (\ref{proprop5.7.3}), we have
\begin{equation}\label{proprop5.7.4}
\int h(|x|)\,\mu_V(dx)=\int_1^{+\infty}\int e^{\lambda
|x|}\,\mu_V(dx)w(\lambda)\,d\lambda<\infty.
\end{equation}

For a fixed $0<\varepsilon<1$ and for $r$ large enough, define
$$\lambda_0(r):=-\frac{1}{c_2}\text{log}\Big[q \Big(\frac{\varepsilon
r}{c_1}\Big)\Big].$$ Thus,
$$e^{-c_2\lambda_0(r)}=q\Big(\frac{\varepsilon
r}{c_1}\Big).$$ According to the definition \eqref{proprop5.7.2} of
$q^{-1}$, for $r$ large enough
$$c_1q^{-1}\big(e^{-c_2\lambda_0(r)}\big)\le \varepsilon r.$$
Hence, for any $\lambda\leqslant \lambda_0(r)$ and $r$ large enough,
$$w(\lambda)\geqslant e^{-\big[(1+c_1)+\varepsilon r\big]\lambda },$$ and so, there exist some constants $c_3$, $c_4>0$ such that
for $r$ large enough,
\begin{equation*}
\begin{split}
 h(r)&\geqslant \int_1^{\lambda_0(r)}\exp\Big\{\big[(1-\varepsilon)r-(1+c_1)\big]\lambda\Big\} \,d\lambda\\
&\geqslant c_3e^{c_4 r\lambda_0(r)}\\
&= c_3\text{exp}\Big\{-\frac{c_4 r}{c_2}\text{log}\Big[q
\Big(\frac{\varepsilon r}{c_1}\Big)\Big]\Big\}.
\end{split}
\end{equation*}
Combining this with (\ref{proprop5.7.4}), we have finished the proof
of the second assertion.
\end{proof}

\bigskip

\noindent{\bf Acknowledgements.} Financial support through the
project ``Probabilistic approach to finite and infinite dimensional
dynamical systems'' funded by the Portuguese Science Foundation
(FCT) (No.\ PTDC/MAT/104173/2008) (for Xin Chen), and National
Natural Science Foundation of China (No.\ 11126350) and the
Programme of Excellent Young Talents in Universities of Fujian (No.\
JA10058 and JA11051) (for Jian Wang) are gratefully acknowledged.

\end{document}